\documentclass[12pt,reqno]{amsart}
\usepackage{latexsym}
\usepackage{mathrsfs}
\usepackage{multirow}
\usepackage{multicol}
\usepackage{amssymb, amsmath}
\usepackage{graphicx}
\usepackage{color}
\usepackage{bbm}
\usepackage{enumerate}
\usepackage{amsmath,bm}
\allowdisplaybreaks[2]
\usepackage[numbers,sort&compress]{natbib}
\usepackage{multirow}
\newtheorem{theo}{\textbf Theorem}[section]
\newtheorem{lem}{\textbf Lemma}[section]
\newtheorem{re}{\textbf Remark}[section]

\newcommand{\lbl}[1]{\label{#1}}

\newcommand{\be}{\begin{equation}}
\newcommand{\ee}{\end{equation}}
\newcommand\bes{\begin{eqnarray}}
\newcommand\ees{\end{eqnarray}}
\newcommand{\bess}{\begin{eqnarray*}}
\newcommand{\eess}{\end{eqnarray*}}

\begin{document}

\title[Two Diffusive SIS Epidemic Models with Mass Action Infection Mechanism]{Novel Spatial Profiles of Population Distribution of Two Diffusive SIS Epidemic Models with Mass Action Infection Mechanism and Small Movement Rate for the Infected Individuals}
\author[R. Peng, Z.-A. Wang, G. Zhang and M. Zhou
]{Rui Peng, Zhi-an Wang, Guanghui Zhang and Maolin Zhou}

\thanks{{\bf R. Peng}: Department of Mathematics, Zhejiang Normal University, Jinhua, Zhejiang, 321004, China. {\bf Email}: {\tt pengrui\_seu@163.com}}

\thanks{{\bf Z.-A. Wang}: Department of Applied Mathematics, The Hong Kong Polytechnic University,
Hung Hom, Kowloon, Hong Kong. {\bf Email}: {\tt mawza@polyu.edu.hk}}

\thanks{{\bf G. Zhang}: School of Mathematics and Statistics, Huazhong University of Science and Technology, Wuhan,
430074, China. {\bf Email}: {\tt guanghuizhang@hust.edu.cn}}

\thanks{{\bf M. Zhou}: Chern Institute of Mathematics and LPMC, Nankai University, Tianjin, 300071, China. {\bf Email}: {\tt zhouml123@nankai.edu.cn}}

\thanks{R. Peng was partially supported by NSF of China (Nos. 12271486, 12171176), Z.-A. Wang was partially supported by the Hong Kong Scholars Program (Project ID P0031250) and an internal grant from the Hong Kong Polytechnic University (Project ID P0031013), G. Zhang was partially supported by NSF of China (No. 12171176, 11971187) and the Fundamental Research Funds for the Central Universities (No. 5003011008), and M. Zhou was partially supported by the Nankai Zhide Foundation and NSF of China (No. 11971498).}

\date{\today}


\begin{abstract} In this paper, we are concerned with two SIS epidemic reaction-diffusion models with mass action infection mechanism of the form $SI$, and study the spatial profile of population distribution as the movement rate of the infected individuals is restricted to be small. For the model with a constant total population number, our results show that the susceptible population always converges to a positive constant which is indeed the minimum of the associated risk function, and the infected population either concentrates at the isolated highest-risk points or aggregates  only on the highest-risk intervals  once  the highest-risk locations contain at least one interval. In sharp contrast, for the model with a varying total population number which is caused by the recruitment of the susceptible individuals  and death of the infected individuals, our results reveal that the susceptible population converges to a positive function which is non-constant unless the associated risk function is constant, and the infected population may concentrate only at some isolated highest-risk points, or aggregate at least in a neighborhood of the highest-risk locations or occupy the whole habitat, depending on the behavior of the associated risk function and even its smoothness at the highest-risk locations. Numerical simulations are performed to support and complement our theoretical findings.
\end{abstract}
\maketitle

\setcounter{equation}{0}

\section{Introduction and existing results}

The outbreak of the novel coronavirus disease 2019 (COVID-19) continues to spread rapidly around the world, and it
has caused tremendous impacts on public health and the global economy. As it is commonly recognized, population movement is a significant factor in the spread of many reported infectious diseases including COVID-19 \cite{Bal,BH,Jia}, and the lockdown and quarantine has turned out to be one of the most effective measures to reduce or even eliminate the infection \cite{KYG,TLL}. On the other hand, the importance of
the population heterogeneity has also been observed in the complicated dynamical behaviour of the transmission of COVID-19 \cite{BBT1,BBT2,Di}.

To gain a deeper understanding of the impact of population movement and heterogeneity on the transmission of epidemic diseases from a mathematically theoretical viewpoint, in the present work we are concerned with two SIS reaction-diffusion systems with mass action infection mechanism in a heterogeneous environment. We aim to study the spatial profile of population distribution as the movement rate of the infected individuals is controlled to be sufficiently small. Such kind of information may be useful for decision-makers to predict the pattern of disease occurrence and henceforth to conduct more effective strategies of disease eradication. The mass action infection mechanism was first proposed in the seminal work of Kermack and McKendrick \cite{KM1}, in which the disease transmission was assumed to be governed by a bilinear incidence function $SI$ (one may also refer to \cite{KM911, KM912, KM913} or \cite{PW}). The systems under consideration in this paper are possibly the simplest yet basic SIS epidemic models.

The first model we will deal with in this work is the following coupled reaction-diffusion equations in one-dimensional space:
\begin{equation}\label{tSIS}
\begin{cases}
S_t-d_{S}S_{xx}=-\beta(x) SI+\gamma(x) I,
\ \ \ &0<x<L,\ \ t>0,\vspace{1mm} \\
I_t-d_{I}I_{xx}=\beta(x) SI-\gamma(x) I,
\ \ \ &0<x<L,\ \ t>0,\vspace{1mm} \\
S_{x}=I_{x}=0,
\ \ \ &x=0,\,L,\ \ t>0,\vspace{1mm} \\
S(x,0)=S_0(x)\ge 0,\ I(x,0)=I_0(x)\ge,\not\equiv 0,\ &0<x<L.
\end{cases}
\end{equation}
Here, $S(x,t)$ and $I(x,t)$ are respectively the population density of the susceptible and infected individuals at position $x \in [0,L]$ and time $t$; the homogeneous Neumann boundary condition means that no population flux crosses the boundary $x=0,\,L$; $d_S$ and $d_I$ are positive constants measuring the motility of susceptible and infected individuals, respectively; and the functions $\beta$ and $\gamma$ are H\"older continuous positive functions in $[0,L]$  representing the disease transmission rate and the disease recovery rate, respectively.

Integrating the sum of the equations of \eqref{tSIS}, combined with the homogeneous Neumann boundary value conditions, we observe that
$$
\int^{L}_{0}(S(x,t)+I(x,t))\,{\rm d}x=\int^{L}_{0}(S_0(x)+I_0(x))\,{\rm d}x=:N,\ \ \ \forall t\ge 0.
$$
Thus, the total population number in \eqref{tSIS} is conserved all the time.

The system \eqref{tSIS} was investigated in the recent works \cite{DW,WZ,WJL}; in particular, when
the movement of either the susceptible or infected population is restricted to be slow, the authors explored the profile of the spatial distribution of the disease modelled by \eqref{tSIS}. The understanding of such a profile amounts to determine the behavior of the so-called endemic equilibrium with respect to the small diffusion rate $d_S$ or $d_I$. The endemic equilibrium of \eqref{tSIS} is a positive steady state solution, which satisfies the following elliptic system:
\begin{equation}\label{SIS}
\begin{cases}
-d_{S}S_{xx}=-\beta(x) SI+\gamma(x) I,
\ \ \ &0<x<L,\vspace{1mm} \\
-d_{I}I_{xx}=\beta(x) SI-\gamma(x) I,
\ \ \ &0<x<L,\vspace{1mm} \\
S_{x}=I_{x}=0,
\ \ \ &x=0,L,\vspace{1mm} \\
\displaystyle
\int^{L}_{0}(S(x)+I(x))\,{\rm d}x=N.
\end{cases}
\end{equation}

According to \cite{DW,WZ,WJL}, if $\min_{x\in[0,L]}\frac{\gamma(x)}{\beta(x)}<\frac{N}{L}$, for any small $d_I>0$,
\eqref{SIS} admits at least one positive solution $(S,I)$, which is called an endemic equilibrium (EE for abbreviation) in terms of epidemiology; moreover, $(S,I)$ satisfies $S,\,I\in C^2([0,L])$ and $S,\,I>0$ on $[0,L]$.

As remarked in \cite{WZ}, it is a challenging problem to study the spatial profile of EE of \eqref{SIS} with respect to the small movement rate $d_I$ of the infected population; in \cite{WJL}, the authors provided a first result in this research direction. Indeed,
they proved the following conclusion.

\begin{theo}\cite[Theorem B]{WJL}\label{mainthm1} Assume that $\min_{x\in[0,L]}\frac{\gamma(x)}{\beta(x)}<\frac{N}{L}$. Then as $d_I\rightarrow0$, the EE $(S,I)$ of \eqref{SIS} satisfies (up to a sequence of $d_I$)
that $S\to \hat S$ uniformly on $[0,L]$, where $\hat S\in C([0,L])$ with
$\min_{[0,L]}\frac{\gamma(x)}{\beta(x)}\leq\hat S(x)\leq\max_{[0,L]}\frac{\gamma(x)}{\beta(x)}$,
and $I\to\mu$ weakly for some Radon measure $\mu$ with nonempty support in the sense of
 \bes\label{wks}
 \int_0^L I(x)\zeta(x){\rm d}x\longrightarrow\int_{[0,L]}\zeta(x)\mu({\rm d}x),\ \ \ \forall \zeta\in C([0,L]).
 \ees
\end{theo}

Obviously, Theorem \ref{mainthm1} does not give a precise description for $\hat S$ and $\mu$ and hence the spatial profile of the susceptible and infected populations remains obscure. From the aspect of disease control, it becomes imperative to know an informative behavior of $\mu$. In this paper, we manage to give a satisfactory result on the profile of $\hat S$ and $\mu$.

In \eqref{tSIS}, some important factors such as the death and recruitment rates of population are ignored so that the total population number is a constant. In order to take into account the death and recruitment rates of population, the following reaction-diffusion epidemic system was proposed in \cite{LPW1}:
 \begin{equation}
 \left\{ \begin{array}{llll}
  S_t-d_SS_{xx}=\Lambda(x)-S-\beta(x)SI+\gamma(x)I,&0<x<L,\,t>0,\\
  I_t-d_II_{xx}=\beta(x)SI-\left[\gamma(x)+\eta(x)\right] I,&0<x<L,\,t>0,\\
  S_{x}=I_{x}=0,&x=0,L,\,t>0,\\
  S(x,0)=S_0(x)\geq0,\,I(x,0)= I_0(x)\geq,\not\equiv 0, &0<x<L.
 \end{array}\right.
 \label{model}
 \end{equation}
The recruitment term of the susceptible population is represented by the function $\Lambda(x)-S$ so that the susceptible is subject to the linear growth/death (\cite{AM, Het2}); $\eta(x)$ accounts for the death rate of the infected. Here, $\Lambda,\,\eta$ are assumed to be positive H\"{o}lder continuous functions on $[0,L]$. All other parameters have the same interpretation as in \eqref{tSIS}.

It is easily seen that the following elliptic problem
\bes\label{dfe}
 -d_SS_{xx}=\Lambda(x)-S,\ \ 0<x<L;\ \  \ S_{x}(0)=S_{x}(L)=0
\ees
admits a unique positive solution $\tilde S$. Then $(\tilde S,0)$ is a unique disease-free equilibrium of \eqref{model}. An EE of \eqref{model} satisfies the following ODE system:
\begin{equation}\label{SIS-2}
\begin{cases}
-d_{S}S_{xx}=\Lambda(x)-S-\beta(x)SI+\gamma(x)I,
\ \ \ &0<x<L,\vspace{1mm} \\
-d_{I}I_{xx}=\beta(x)SI-\left[\gamma(x)+\eta(x)\right] I,
\ \ \ &0<x<L,\vspace{1mm} \\
S_{x}=I_{x}=0,
\ \ \ &x=0,L.
\end{cases}
\end{equation}
As one of the main results of \cite{LPW1},  the following conclusion on the profile of EE of \eqref{SIS-2} with respect to small $d_I$ was established.

\begin{theo}\cite[Theorem 3.2]{LPW1}\label{mainthm2} Assume that the set $\{x\in [0,L]: \ \beta(x)\tilde S(x)>\gamma(x)+\eta(x)\}$ is non-empty. As $d_I\rightarrow0$, then any EE $\left(S,I\right)$
of \eqref{SIS-2} satisfies (up to a subsequence of $d_I$) that
$S\to \hat S\ \ \mbox{uniformly on}\ [0,L],$ where $\hat S\in C([0,L])$ and $\hat S>0$ on $[0,L]$, and
$\int_0^L I{\rm d}x\to \hat I$ for some positive constant $\hat I$.
\end{theo}

As in Theorem \ref{mainthm1}, Theorem \ref{mainthm2} does not characterize the precise distribution of the susceptible and infected populations. In this paper, we will also provide a clear picture of the population distributions for \eqref{SIS-2} as the movement rate $d_I$ tends to zero. It turns out that the spatial profiles of the disease distribution modelled by \eqref{SIS} and \eqref{SIS-2} are rather different.

The rest of paper is organized as follows. In section 2, we state the main theoretical results,
and section 3 is devoted to their proofs. In section 4, we carry out the numerical simulations and discuss the implications of our results in terms of disease control. In the appendix, we recall some known facts which will be used in the paper.

\section{Statement of main results}
In this section, we state the main findings of this paper on models \eqref{SIS} and \eqref{SIS-2}. To proceed, we underline some terminologies frequently used throughout the paper. For model \eqref{SIS}, we call $\frac{\gamma(x)}{\beta(x)}$  the {\it risk function}, and call each element of the set
$\big\{x\in[0,L]:\ \frac{\gamma(x)}{\beta(x)}=\min_{x\in[0,L]}\frac{\gamma(x)}{\beta(x)}\big\}$
 the {\it highest-risk point} (or {\it location}). Similarly, for model \eqref{SIS-2}, we call $\frac{\gamma(x)+\eta(x)}{\beta(x)}$  the {\it risk function}, and call each element of the set
$\big\{x\in[0,L]:\ \frac{\gamma(x)+\eta(x)}{\beta(x)}=\min_{x\in[0,L]}\frac{\gamma(x)+\eta(x)}{\beta(x)}\big\}$  the {\it highest-risk point} (or {\it location}).

\subsection{Results for model \eqref{SIS}}
For the sake of convenience, we set
 $$
k(x)=\frac{\gamma(x)}{\beta(x)},\ \ \ \ k_{min}=\min_{x\in[0,L]} k(x),$$
and
 $$\Theta_{k}=\big\{x\in[0,L]:\ k(x)=k_{min}\big\}.
 $$
We note that when the risk function $k(x)=k$ is a positive constant, it follows from \cite{WJL} that
$S(x)\equiv k$ is a constant, and in turn by the equation of $I$, we immediately see that $I=\frac{N}{L}-k$ is also a positive constant provided that $k<\frac{N}{L}$. In what follows, we do not consider such a trivial case and assume that $k(x)$ is non-constant on $[0,L]$.

We now state our main result on the asymptotic behavior of any EE $(S,I)$ of \eqref{SIS} as $d_I\rightarrow0$ as follows.

\begin{theo}\label{th2.1} Assume that $k(x)$ is non-constant and $k_{min}<\frac{N}{L}$. Then as $d_I\rightarrow0$, the EE $(S,I)$ of \eqref{SIS} satisfies
 \bes\label{S-limt}
 S(x)\to k_{min}\ \ \ \mbox{uniformly for}\ x\in[0,L].
 \ees
The following assertions hold for the asymptotic behavior of $I$.

\begin{enumerate}
  \item[{\rm(i)}] If $\Theta_{k}=\{x_0\}$, then we have
 $$
 I(x)\to(N-Lk_{min})\delta(x_0)\ \ \mbox{weakly in the sense of\ \eqref{wks}},
 $$
 where $\delta(x_0)$ is the Dirac measure centered at $x_0$. Moreover, $I(x)\to0$ locally uniformly in $[0,L]\setminus\{x_0\}$.

   \item[{\rm(ii)}] If $\Theta_{k}=[\varrho_1,\varrho_2]$ for some $0<\varrho_1<\varrho_2<L$, then we have
   $$
   \mbox{$I(x)\to0$\ \ \ uniformly on $[0,\varrho_1]\cup[\varrho_2,L]$,}
   $$
and
 $$
 I(x)\to \hat I(x)\ \ \ \ \mbox{uniformly for}\ x\in[\varrho_1,\varrho_2],
 $$
 where $\hat I\in C^2([\varrho_1,\varrho_2])$, $\hat I>0$ in $(\varrho_1,\varrho_2)$, and $\hat I$ is the unique positive solution of
\begin{equation}\label{limit-I}
\begin{cases}
-\hat I_{xx}=\frac{\beta(x)}{d_S}(\hat a-\hat I)\hat I,
\ \ \ \  \ \varrho_1<x<\varrho_2,\vspace{1mm} \\
\hat I=0,\ \ \ \ \ \ \ \ \ \ \  \ \ \ \  \ \ \ \  \ \ \ \ \ x=\varrho_1,\,\varrho_2,\vspace{1mm} \\
 \displaystyle
\int_{\varrho_1}^{\varrho_2}\hat I\,{\rm d}x=N-Lk_{min},
\end{cases}
\end{equation}
where the positive constant $\hat a$ is uniquely determined by the integral constraint in \eqref{limit-I}.

\end{enumerate}

\end{theo}

Regarding Theorem \ref{th2.1}, we would like to make some comments in order as follows.

\begin{re}\label{re2.2}\noindent In addition to the two cases treated in Theorem \ref{th2.1}, we can handle some  more general cases.
In particular, we would like to make the following comments.

\begin{itemize} \item[\rm{(i)}] If the set $\Theta_{k}$ contains only finitely many isolated points, say $\{x_i\}_{i=1}^j$ for some $j\geq2$, then one can slightly modify the proof of Theorem \ref{th2.1}(i) to show that
$S\to k_{min}$ uniformly on $[0,L]$, and $I\to0$ locally uniformly in $[0,L]\setminus(\{x_i\}_{i=1}^j)$, and
$$
 I(x)\to\sum_{i=1}^jc_i\delta(x_i)\ \ \mbox{weakly in the sense of\ \eqref{wks}},
 $$
where $\delta(x_i)$ is the Dirac measure centered at $x_i$ and the nonnegative constants $c_i$ fulfill $\sum_{i=1}^jc_i=N-Lk_{min}$.
Nevertheless, we can not determine the exact values of $c_i$; in other words, as $d_I\to0$, it is unclear to us whether $I$ concentrates at all $x_i\ (1\leq i\leq j)$ or only some of them. The numerical results suggest that the former alternative holds; see Figure \ref{fig1} in section 4.

\item[\rm{(ii)}] If the set $\Theta_{k}$ contains at least one proper interval of $[0,L]$,
by adapting the argument of Theorem \ref{th2.1}(ii), we can show that $S\to k_{min}$ uniformly on $[0,L]$, and $I\to\hat I$ uniformly on $[0,L]$ with
$$\hat I=0\ \ \mbox{ on}\ [0,L]\setminus\Theta_{k},\ \ \ \ \int_{\Theta_{k}}\hat I\ {\rm d}x=N-Lk_{min}.
$$
In particular, if $\Theta_{k}=\Big(\bigcup_{i=1}^{j_*}[\underline \varrho_i,\,\overline \varrho_i]\Big)\bigcup\Big(\bigcup\{x_i\}_{i=0}^{j^*}\Big)$ for some $j_*\geq1,\,j^*\geq0$,
then we can prove that
$$\hat I=0\ \ \mbox{ on}\ [0,L]\setminus(\bigcup_{i=1}^{j_*}(\underline \varrho_i,\,\overline \varrho_i)),$$
and in $(\underline \varrho_i,\,\overline \varrho_i)\ (1\leq i\leq j_*)$, either $\hat I=0$ or $\hat I>0$. Without loss of generality, assuming that $\hat I(x)>0$ for $x\in\bigcup_{i=1}^{\hat {j_*}}(\underline \varrho_i,\,\overline \varrho_i)$ for some $1\leq \hat {j_*}\leq j_*$, then in each such $(\underline \varrho_i,\,\overline \varrho_i)$, we can conclude that $\hat I$ solves
\begin{equation}\nonumber
\begin{cases}
-\hat I_{xx}=\frac{\beta(x)}{d_S}(\hat a-\hat I)\hat I,\ \ & \underline \varrho_i<x<\overline \varrho_i,\vspace{1mm} \\
\hat I=0,\ \ &x=\underline \varrho_i,\,\overline \varrho_i,
\end{cases}
\end{equation}
where the positive constant $\hat a$ is uniquely determined by
 $$
 \sum_{i=1}^{\hat {j_*}}\int_{\underline \varrho_i}^{\overline \varrho_i}\hat I\,{\rm d}x=N-Lk_{min}.
 $$
However, it seems rather challenging to prove whether $\hat I$ is positive on all intervals $(\underline \varrho_i,\,\overline \varrho_i)\ (1\leq i\leq j_*)$ or only on some of them. Our numerical results suggest that the former alternative holds; see Figure \ref{fig2} in section 4.

\item[\rm{(iii)}] The assertion in (ii) above suggests that if the highest-risk locations contain at least one interval, then the disease can not stay on any possible isolated highest-risk points once the infected individuals move slowly.
\end{itemize}
\end{re}

\begin{re}\label{re2.1}\noindent In the case (ii) of Theorem \ref{th2.1}, if $\varrho_1=0$ (or $\varrho_2=L$), the results of Theorem \ref{th2.1} still hold true if we replace the Dirichlet boundary condition of $\hat I$ in \eqref{limit-I} at $\varrho_1=0$ (or $\varrho_2=L$) by the Neumann boundary condition $\hat I_x(0)=0$ (or $\hat I_x(L)=0$). A similar remark applies to the case discussed in Remark \ref{re2.2}(ii) above.
\end{re}

\begin{re}\label{re2.1aa}\noindent After this paper was finished, we noticed the work \cite{CS} in which the authors derived \eqref{S-limt} and the convergence of the $I$-component in the case (i) of Theorem \ref{th2.1} in any spatial dimension in a more general setting; see Theorem 2.5(i) there. However, their result does not establish the convergence of the $I$-component within $\Theta_{k}$ in the case (ii) of Theorem \ref{th2.1} nor in the more general case mentioned by Remark \ref{re2.2}; on the other hand, our proof of \eqref{S-limt} and the convergence of the $I$-component outside of $\Theta_{k}$ is rather different from that of \cite{CS}.

\end{re}

\subsection{Results for model \eqref{SIS-2}} We now turn to system \eqref{SIS-2}. For the sake of simplicity, we assume that $\Lambda$ in \eqref{SIS-2} is a positive constant, and also denote
$$
h(x)=\frac{\gamma(x)+\eta(x)}{\beta(x)},\ \ \ \  h_{min}=\min_{x\in[0,L]}h(x),$$
and
$$
\Theta_h=\big\{x\in[0,L]:\ h(x)=h_{min}\big\}.
$$
Clearly, $\tilde S(x)=\Lambda$. We also enhance the existence condition of EE of \eqref{SIS-2} in Theorem \ref{mainthm2} by imposing the following condition:
\bes\label{cond}
\Lambda>h(x)\ \ \ \mbox{for all}\ \ x\in[0,L].
\ees

\vskip5pt Now we can state our main findings on the asymptotic behavior of any EE $(S,I)$ of \eqref{SIS-2} as $d_I\rightarrow0$. The first result reads as follows.

\begin{theo}\label{th2.2} Assume that \eqref{cond} holds.
As $d_I\rightarrow0$, then any EE $\left(S,I\right)$
of \eqref{SIS-2} satisfies (up to a subsequence of $d_I$) that
$S\to \hat S\ \ \mbox{uniformly on}\ [0,L]$, and $I\to\mu$ weakly in the sense of \eqref{wks}, where
$\mu$ is some Radon measure and $\hat S$ solves
weakly in $W^{1,2}(0,L)$ the free boundary problem:
\bes\label{th2.2-1}
-d_S\hat S_{xx}=\Lambda-\hat S-\eta(x)\mu(\{x\})\big|_{\{x\in[0,\,L]:\ \hat S(x)=h(x)\}},\ \ \ x\in(0,L).
\ees
Here, $\mu(\{x\})\big|_{\{x\in[0,\,L]:\ \hat S(x)=h(x)\}}$ is the restriction of $\mu$ on the set $\{x\in[0,L]:\ \hat S(x)=h(x)\}$; otherwise, $\mu(\{x\})=0$. Moreover we have the following properties for $\mu$ and $\hat S$.

\begin{enumerate}

  \item[{\rm(i)}] The Radon measure $\mu$ satisfies
\bes\label{th2.2-3}
\mu(\{x\in[0,L]:\ \hat S(x)\not=h(x)\})=0, \ \ \mu(\{x\in[0,L]:\ \hat S(x)=h(x)\})>0.
\ees

  \item[{\rm(ii)}] The function $\hat S\in C([0,L])$ satisfies
\bes\label{th2.2-2}
h_{min}\leq\hat S(x)\leq h(x),\ \ \forall x\in[0,L],
\ees
\bes\label{th2.2-2a}
\Theta_h\subset\big\{x\in[0,L]:\ \hat S(x)=h(x)\big\};
\ees

If $x_1,\,x_2\in\Theta_h$ with $x_1<x_2$ and $(x_1,x_2)\cap\Theta_h=\emptyset$, then
\bes\label{th2.2-2b}
h_{min}<\hat S(x),\ \ \ \forall x\in(x_1,x_2).
\ees


\end{enumerate}
\end{theo}

Theorem \ref{th2.2} asserts that $\hat S$ touches $h$ at all highest-risk points.
In what follows, our goal is to examine the properties $\hat S$ for some specific risk function $h$, which in turn provides us with a more precise description of the profile of $\mu$. Indeed, we can obtain the following result for \eqref{SIS-2}.

\begin{theo}\label{th2.3} Let $\hat S$ and $\mu$ be given as in Theorem \ref{th2.2}.
Assume that $h\in C^2([0,L])$ and \eqref{cond} holds. The following assertions hold.

\begin{enumerate}
  \item[{\rm(i)}] If $-d_Sh_{xx}\leq\Lambda-h$ in $(0,L)$, $h_x(0)\geq0$ and $h_x(L)\leq0$,
  then we have
  \bes\label{th2.3-a}
  \hat S(x)=h(x),\ \ \ \ \forall x\in[0,L],
  \ees
  \bes\label{th2.3-b}
  \mu(\{x\})=\frac{\Lambda-h(x)+d_Sh_{xx}(x)}{\eta(x)},\ \ \mbox{ a.e. for}\ \, x\in(0,L).
  \ees

 \item[{\rm(ii)}] If $h_x$ is non-decreasing on $[0,L]$ and $\Theta_h=\{\tau_0\}$ for some $0\leq\tau_0\leq L$,
 then the following assertions hold.

 \begin{enumerate}
  \item[{\rm(a)}] When $0<\tau_0<L$, we have
  \bes\label{th2.3-c}
      \hat S(x)=h(x),\ \ \forall x\in[\tau_1,\tau_2],
   \ees
   and in $[0,\tau_1)\cup(\tau_2,L]$, $\hat S<h$ satisfies
\begin{equation}\label{th2.3-d}
\begin{cases}
 -d_S\hat S_{xx}(x)=\Lambda-\hat S,\ \  \ x\in(0,\tau_1)\cup(\tau_2,L),\vspace{1mm} \\
\hat S_x(0)=0,\ \ \hat S_x(L)=0,\vspace{1mm} \\
\hat S(\tau_1)=h(\tau_1),\ \ \hat S(\tau_2)=h(\tau_2),
\end{cases}
\end{equation}
   and $\mu$ satisfies
     \bes\label{th2.3-e}
      \mu(\{x\})=\frac{\Lambda-h(x)+d_Sh_{xx}(x)}{\eta(x)},\ \ \mbox{ a.e. for}\ \, x\in(\tau_1,\tau_2),
      \ees
      \bes\label{th2.3-f}
      \mu(\{x\})=0, \ \ \ \ \forall x\in[0,\tau_1)\cup(\tau_2,L],
      \ees
where the numbers $\tau_1,\,\tau_2$ with $0<\tau_1<\tau_0<\tau_2<L$ are uniquely determined by
\bes\label{th2.3-g}
\frac{e^{2d_S^{-1/2}\tau_1}-1}{e^{2d_S^{-1/2}\tau_1}+1}=-\frac{d_S^{1/2}h_x(\tau_1)}{\Lambda-h(\tau_1)},\ \ \ \  \
\frac{e^{2d_S^{-1/2}(\tau_2-L)}-1}{e^{2d_S^{-1/2}(\tau_2-L)}+1}=-\frac{d_S^{1/2}h_x(\tau_2)}{\Lambda-h(\tau_2)}.
\ees

  \item[{\rm(b)}] When $\tau_0=L$, then we have the following assertions.
     \begin{enumerate}
  \item[{\rm(b-1)}] If $\frac{e^{2Ld_S^{-1/2}}-1}{e^{2Ld_S^{-1/2}}+1}>-\frac{d_S^{1/2}h_x(L)}{\Lambda-h(L)}$,
then \eqref{th2.3-c} and \eqref{th2.3-e} hold with $[\tau_1,\tau_2]$ replaced by $[\tau_1,L]$, $\mu([0,\tau_1))=0$, and on $[0,\tau_1]$, $\hat S$ satisfies
\begin{equation}\label{3.33-a}
\begin{cases}
 -d_S\hat S_{xx}(x)=\Lambda-\hat S,\ \ \ x\in(0,\tau_1),\vspace{1mm} \\
\hat S_x(0)=0,\ \ \hat S(\tau_1)=h(\tau_1),
\end{cases}
\end{equation}
 where $0<\tau_1<L$ is uniquely determined by the first equation in \eqref{th2.3-g}.

  \item[{\rm(b-2)}] If $\frac{e^{2Ld_S^{-1/2}}-1}{e^{2Ld_S^{-1/2}}+1}\leq-\frac{d_S^{1/2}h_x(L)}{\Lambda-h(L)}$,
then $\hat S$ is the unique positive solution of
\begin{equation}\label{3.33-b}
\begin{cases}
 -d_S\hat S_{xx}(x)=\Lambda-\hat S,\ \ \ x\in(0,L),\vspace{1mm} \\
\hat S_x(0)=0,\ \ \hat S(L)=h(L),
\end{cases}
\end{equation}
and $\mu$ satisfies
\bes\label{th2.3-h}
\mu([0,L))=0,\ \ \ \mu(\{L\})=\frac{\Lambda L-\int_0^L\hat S(x){\rm d}x}{\eta(L)}.
\ees

\end{enumerate}

\item[{\rm(c)}] When $\tau_0=0$, then we have the following assertions.
     \begin{enumerate}
  \item[{\rm(c-1)}] If $\frac{e^{2Ld_S^{-1/2}}-1}{e^{2Ld_S^{-1/2}}+1}>\frac{d_S^{1/2}h_x(0)}{\Lambda-h(0)}$,
then \eqref{th2.3-c} and \eqref{th2.3-e} hold with $[\tau_1,\tau_2]$ replaced by $[0,\tau_2]$, $\mu((\tau_2,L])=0$, and on $[\tau_2,L]$, $\hat S$ satisfies
\begin{equation}\label{3.33-c}
\begin{cases}
 -d_S\hat S_{xx}(x)=\Lambda-\hat S,\ \ \ x\in(\tau_2,L),\vspace{1mm} \\
\hat S_x(L)=0,\ \ \hat S(\tau_2)=h(\tau_2),
\end{cases}
\end{equation}
 where $0<\tau_2<L$ is uniquely determined by the second equation in \eqref{th2.3-g}.

  \item[{\rm(c-2)}] If $\frac{e^{2Ld_S^{-1/2}}-1}{e^{2Ld_S^{-1/2}}+1}\leq\frac{d_S^{1/2}h_x(0)}{\Lambda-h(0)}$,
then $\hat S$ is the unique positive solution of
\begin{equation}\label{3.33-d}
\begin{cases}
 -d_S\hat S_{xx}(x)=\Lambda-\hat S,\ \ \ x\in(0,L),\vspace{1mm} \\
\hat S_x(L)=0,\ \ \hat S(0)=h(0),
\end{cases}
\end{equation}
and $\mu$ satisfies
\bes\label{th2.3-i}
\mu((0,L])=0,\ \ \ \mu(\{0\})=\frac{\Lambda L-\int_0^L\hat S(x){\rm d}x}{\eta(0)}.
\ees

\end{enumerate}

\end{enumerate}

  \item[{\rm(iii)}] If $h_x$ is non-decreasing on $[0,\varrho_1]\cup[\varrho_2,L]$ and $\Theta_h=[\varrho_1,\varrho_2]$ for some $0<\varrho_1<\varrho_2<L$, then all the assertions in {\rm (ii)-(a)} above hold, where
 the numbers $\tau_1,\,\tau_2$ satisfying $0<\tau_1<\varrho_1<\varrho_2<\tau_2<L$ are uniquely determined by
\eqref{th2.3-g}.
\end{enumerate}

\end{theo}

For model \eqref{SIS}, our result shows that the infected population concentrates or aggregates  only at the highest-risk locations.
In sharp contrast, for model \eqref{SIS-2}, our result suggests that the disease will occupy a neighborhood
of the interior highest-risk locations or even occupy the whole habitat $[0,L]$, or concentrates only at the boundary highest-risk location, depending on the risk function $h$. More detailed discussions on the implications of our theoretical results, along with numerical simulations, will be given in section 4.

We would like to make some remarks on Theorem \ref{th2.3} as follows.

\begin{re}\label{re-th2.3-a} It is worth mentioning that all the statements in Theorem \ref{th2.3}
except the expression \eqref{th2.3-e} for the Radon measure $\mu$ remain true provided that the risk function $h\in C^1([0,L])$. Such a comment also applies to Lemmas \ref{l2.3}-\ref{l2.4a} in the forthcoming section.
\end{re}

\begin{re}\label{re-th2.3-b}
\begin{enumerate}
  \item[{\rm(i)}] It is clear that Theorem \ref{th2.3}(i) holds if $h<\Lambda$ is a constant or more generally $h$ is a unique solution to the following problem:
\begin{equation}\nonumber
\begin{cases}
 -d_Sh_{xx}=\Lambda-h,\ \  \ x\in(0,L),\vspace{1mm} \\
h(0)=\sigma_1,\ \ h(L)=\sigma_2,
\end{cases}
\end{equation}
where $0<\sigma_1,\,\sigma_2<\Lambda$.

When $h_x(0)>0$, the change of the derivatives from $S_x(0)=0$ to $\hat S_x(0)=h_x(0)>0$ would suggest that $I$ should experience the concentration phenomenon at $x=0$ (that is, $I(0)\to\infty$) as $d_I\to0$. The same remark applies to the case of $h_x(L)<0$.

  \item[{\rm(ii)}] In contrast to Theorem \ref{th2.3}(i), it is easily seen that $\hat S\not\equiv h$ on $[0,L]$ provided that $-d_Sh_{xx}(x^*)>\Lambda-h(x^*)$ for some $x^*\in(0,L)$.

  \item[{\rm(iii)}] Clearly, the assertions of Theorem \ref{th2.3}(ii)-(b1) hold if $h_x(L)=0$ and the assertions of  Theorem \ref{th2.3}(ii)-(c1) hold if $h_x(0)=0$.

   \item[{\rm(iv)}] In a general case that $\Theta_{h}$ contains an interior isolated point and $h_x$ is non-decreasing in a neighbourhood of such a point, we can conclude that \eqref{th2.3-a} and  \eqref{th2.3-b} hold in some neighbourhood of this point; if $\Theta_{h}$ contains an interval, a similar conclusion also holds. See Lemma \ref{l2.3} and Lemma \ref{l2.3a} below.

\end{enumerate}

\end{re}

\section{Proof of main results:\ Theorems \ref{th2.1}, \ref{th2.2} and \ref{th2.3}}
\setcounter{equation}{0}

This section is devoted to the proof of Theorems \ref{th2.1}, \ref{th2.2} and \ref{th2.3}.

\subsection{Proof of Theorem \ref{th2.1}}

In this subsection, we present the proof of Theorem \ref{th2.1}.

\begin{proof}[Proof of Theorem \ref{th2.1}] First of all, we recall that for any EE $(S,I)$ of \eqref{SIS}, from \cite{WJL} (see (3.3) there), the following holds:
 \bes
 k_{min}\leq S(x)\leq\max_{[0,L]} k(x),\ \ \forall x\in[0,L].
 \label{2.1}
 \ees

By the positivity of $I$ and the uniqueness of the principal eigenvalue,
it is clear from the equation of $I$ that
$$
\lambda_1(d_I,\gamma-\beta S)=0,\ \  \forall d_I>0,
$$
where $\lambda_1(d_I,\gamma-\beta S)$ is defined as in the appendix. Using Theorem \ref{mainthm1},
as $d_I\to0$ (up to a subsequence), we see that $S\to \hat S$ uniformly on $[0,L]$ for some positive function $\hat S$. Hence, by Lemma \ref{a2.1} in the appendix and the continuous dependence of the principal eigenvalue on the weight function $\gamma-\beta S$, we have
$$
0=\lim_{d_I\to0}\lambda_1(d_I,\gamma-\beta S)=\min_{x\in[0,L]}[\gamma(x)-\beta(x)\hat S(x)].
$$
This obviously implies that
\bes\label{3.1}
\hat S(x)\leq k(x),\ \ \forall x\in[0,L]\ \ \mbox{and}\ \ \hat S(y_0)=k(y_0)
\ees
for some $y_0\in[0,L]$.

From Theorem \ref{mainthm1}, we recall that $I\to\mu$ weakly
for some Radon measure $\mu$ with $\mu([0,L])>0$ in the following sense
\bes\label{3.2}
\int_0^LI(x)\zeta(x){\rm d}x\to\int_0^L\zeta(x)\mu({\rm d}x),\ \ \forall \zeta\in C([0,L]),\ \ \mbox{as}\ d_I\to0.
\ees

We now integrate the first equation in \eqref{SIS} by parts over $[0,L]$ and use the boundary conditions to deduce that
\bes\label{3.3}
\int_0^L[\beta(x)S(x)-\gamma(x)]I(x){\rm d}x=0,\ \ \ \forall d_I>0.
\ees
Letting $d_I\to0$ in \eqref{3.3}, combined with \eqref{3.2} and the fact that $S\to \hat S$
uniformly on $[0,L]$ as $d_I\to0$, we infer that
\bes\label{3.4}
\int_{[0,L]}[\beta(x)\hat S(x)-\gamma(x)]\mu({\rm d}x)=0,
\ees
which, together with \eqref{3.1}, gives
\bes\nonumber
\int_{\{x\in[0,L]:\ \hat S(x)<k(x)\}}\beta(x)[\hat S(x)-k(x)]\mu({\rm d}x)=\int_{[0,L]}\beta(x)[\hat S(x)-k(x)]\mu({\rm d}x)=0.
\ees
As a result, we find that
\bes\label{3.5}
\mu(\{x\in[0,L]:\ \hat S(x)<k(x)\})=0
\ees
and
\bes\label{3.6}
\mu(\{x\in[0,L]:\ \hat S(x)=k(x)\})=\mu([0,L])>0.
\ees

In view of \eqref{3.3} and $\int^{L}_{0}(S(x)+I(x))\,{\rm d}x=N$, for any $d_I>0$ we have
\bes\label{3.7}
\int_0^LS(x)I(x){\rm d}x\leq\frac{1}{\min_{[0,L]}\beta(x)}\int_0^L\gamma(x)I(x){\rm d}x\leq\frac{\max_{[0,L]}\gamma(x)}{\min_{[0,L]}\beta(x)}N,\ \ \ \forall d_I>0.
\ees
Then, applying the $L^1$-theory for elliptic equation (see Lemma \ref{a2.2} in the appendix) to
the $S$-equation, one sees that for any $1\leq r<\infty$,
\bes\label{3.8}
\|S\|_{W^{1,r}(0,L)}\leq C,\ \ \ \forall d_I>0.
\ees
Hereafter, $C$ or $C(\epsilon)$ is a positive constant independent of $d_I>0$ but may be different from place to place.
Taking $r=2$ in \eqref{3.8}, we note that $W^{1,2}(0,L)$ is a Hilbert space and $W^{1,2}(0,L)$ is compactly embedded to $C([0,L])$.
Thus, we may assume that $S\to\hat S$ weakly in $W^{1,2}(0,L)$ and $S\to\hat S$ uniformly on $[0,L]$ as $d_I\to0$. Now, for any $\zeta\in W^{1,2}(0,L)$
(and so $\zeta\in C([0,L])$), we get from the $S$-equation that
\bes\label{3.9}
d_S\int_0^LS_x(x)\zeta_x(x){\rm d}x=\int_0^L[-\beta(x)S(x)+\gamma(x)]I(x)\zeta(x){\rm d}x,\ \ \ \forall d_I>0.
\ees
By virtue of \eqref{3.2}, \eqref{3.5} and \eqref{3.6}, we can send $d_I\to0$ in \eqref{3.9} to obtain
 \bes\nonumber
 d_S\int_0^L\hat S_x(x)\zeta_x(x){\rm d}x=0,\ \ \forall \zeta\in W^{1,2}(0,L).
\ees
This means that $\hat S$ is a weak (and then a classical) solution of
 \bes\nonumber
 -u_{xx}(x)=0,\ \ x\in(0,L);\ \ \ \ u_x(0)=u_x(L).
\ees
Consequently, $\hat S$ must be a positive constant. It then follows from \eqref{3.1} that
$\hat S=k_{min}$, and so $S(x)\to k_{min}$ uniformly on $[0,L]$.

In the sequel, we are going to determine the limit of $I$. We first consider case (i):
$\Theta_{k}=\{x_0\}$ is a singleton. By what was proved above, it is easily seen that
$$
I(x)\to(N-Lk_{min})\delta(x_0)\ \ \mbox{ weakly in the sense of \eqref{wks}},
$$
where $\delta(x_0)$ is the Dirac measure centered at $x_0$.

It remains to show $I(x)\to0$ locally uniformly in $[0,L]\setminus\{x_0\}$. We only consider the case of $x_0\in(0,L)$, and the case $x_0=0$ or $L$ can be handled similarly. Since $S(x)\to k_{min}$ uniformly on $[0,L]$, by the definition of $k_{min}$, we know from the $I$-equation that, given small $\epsilon>0$, $I_{xx}>0$ on $[0,x_0-\epsilon]\cup[x_0+\epsilon,L]$ as long as $d_I$ is small enough. As $I_x(0)=I_x(L)=0$, $I$ is increasing in $[0,x_0-\epsilon]$ while is decreasing in $[x_0+\epsilon,L]$. Thus, due to the arbitrariness of $\epsilon$, it readily follows from \eqref{3.5} that $I(x)\to0$ locally uniformly in $[0,x_0)\cup(x_0,L]$, as claimed.

\vskip 10pt We next consider case (ii): $\Theta_{k}=[\varrho_1,\varrho_2]\subset(0,L)$.
First of all, we can assert that $I(x)\to0$ locally uniformly in $[0,L]\setminus{[\varrho_1,\varrho_2]}$ by a similar argument as in case (i). In what follows, we will analyze the limiting behavior of $I$ in the interval $[\varrho_1,\varrho_2]$.
To this end, let us introduce the following function
 $$
 w(x)=\frac{S(x)-k_{min}}{d_I},\ \ \ x\in[0,L].
 $$
Due to \eqref{2.1}, $w\geq0$ on $[0,L]$. In addition, by our assumption, one notices that $w$ solves
\bes\label{3.10}
-d_Sw_{xx}(x)=-\beta(x)Iw,\ \ \ x\in[\varrho_1,\varrho_2],
\ees
and $I$ satisfies
 \bes\label{3.11}
-I_{xx}(x)=\beta(x)wI,\ \ \ x\in[\varrho_1,\varrho_2].
\ees
Since $\int_0^L I(x){\rm d}x\leq N$, for any small $\epsilon>0$, Lemma \ref{a2.3}(b) in the appendix can be applied to \eqref{3.10} to assert that
 \bes\label{3.12}
\max_{x\in[\varrho_1+\epsilon,\varrho_2-\epsilon]}w(x)\leq C(\epsilon)\min_{x\in[\varrho_1+\epsilon,\varrho_2-\epsilon]}w(x).
\ees

We now claim that $w$ is uniformly bounded on $[\varrho_1+\epsilon,\varrho_2-\epsilon]$ for all small $d_I>0$. Otherwise,
there is a sequence of $d_I$, labelled by itself for simplicity, such that the corresponding solution sequence $\{(w,I)\}$ satisfies
 \bes\label{3.12-a}
\max_{x\in[\varrho_1+\epsilon,\varrho_2-\epsilon]}w(x)\to\infty,\ \  \mbox{ as $d_I\to0$.}
\ees
By \eqref{3.12}, $w\to\infty$ uniformly on $[\varrho_1+\epsilon,\varrho_2-\epsilon]$ as $d_I\to0$. To produce a contradiction, let us denote $\lambda_1^\mathcal{D}$ to be the principal eigenvalue of the following eigenvalue problem with Dirichlet boundary conditions:
\begin{equation}\label{3.13}
\begin{cases}
 -\varphi_{xx}=\lambda\varphi,\ \ x\in(\varrho_1+\epsilon,\varrho_2-\epsilon)\vspace{1mm} \\
\varphi(\varrho_1+\epsilon)=\varphi(\varrho_2-\epsilon)=0.
\end{cases}
\end{equation}
Apparently, $\lambda_1^\mathcal{D}>0$. For all small $d_I>0$, by \eqref{3.12-a} we may assume that
$$\beta(x)w(x)>2\lambda_1^\mathcal{D}\  \ \mbox{ on $[\varrho_1+\epsilon,\varrho_2-\epsilon]$.}
$$
Thus, it follows from \eqref{3.11} that $I\in C^2([0,L])$ is a positive and strict supersolution of the following operator in the sense of \cite[Definition 2.1]{PZ2}:
\begin{equation}\nonumber
\begin{cases}
\mathcal{L}u:=-u_{xx}-2\lambda_1^\mathcal{D}u,\ \ x\in(\varrho_1+\epsilon,\varrho_2-\epsilon),\ \ \forall u\in C^2([0,L]),\vspace{1mm} \\
u(\varrho_1+\epsilon)=u(\varrho_2-\epsilon)=0.
\end{cases}
\end{equation}
By means of \cite[Proposition 2.1]{PZ2}, the principal eigenvalue, denoted by $\tilde\lambda_1^\mathcal{D}$, of the eigenvalue problem
\begin{equation}\nonumber
\begin{cases}
\mathcal{L}\varphi_{xx}=\lambda\varphi,\ \ \ \ x\in(\varrho_1+\epsilon,\varrho_2-\epsilon),\vspace{1mm} \\
\varphi(\varrho_1+\epsilon)=\varphi(\varrho_2-\epsilon)=0
\end{cases}
\end{equation}
satisfies $\tilde\lambda_1^\mathcal{D}>0.$

On the other hand, the uniqueness of the principal eigenvalue of problem \eqref{3.13} implies $\tilde\lambda_1^\mathcal{D}+2\lambda_1^\mathcal{D}=\lambda_1^\mathcal{D}$, and so $\tilde\lambda_1^\mathcal{D}=-\lambda_1^\mathcal{D}<0$, leading to a contradiction. The previous claim is thus verified. Due to the arbitrariness of $\epsilon$, we have shown that $w$ is locally uniformly bounded in $(\varrho_1,\varrho_2)$ with respect to all small $d_I>0$.

Furthermore, by Lemma \ref{a2.2} in the appendix, it is easy to see from \eqref{3.11} that $I$ is locally uniformly bounded in $(\varrho_1,\varrho_2)$ independent of all small $d_I>0$. The standard regularity theory for elliptic equations can be applied to \eqref{3.10} and \eqref{3.11}, respectively to deduce that $w$ and $I$ are locally bounded (independent of small $d_I$) in $(\varrho_1,\varrho_2)$ in the usual $C^{2+\alpha}$-norm for some $\alpha\in(0,1)$. Then, by a diagonal argument, we may assume that
$$
(w,I)\to(\hat w,\hat I)\ \ \mbox{ in}\ C_{loc}^2(\varrho_1,\varrho_2),\ \ \mbox{ as}\ \, d_I\to0.
$$
Clearly, by \eqref{3.11}, $(\hat w,\hat I)$ satisfies
 \bes\label{3.11a}
-\hat I_{xx}(x)=\beta(x)\hat w\hat I,\ \ \ x\in(\varrho_1,\varrho_2).
\ees
Furthermore, by adding \eqref{3.10} and \eqref{3.11}, one easily sees that $(\hat w,\hat I)$ solves
$$-(d_S\hat w+\hat I)_{xx}=0\  \  \mbox{ in $(\varrho_1,\varrho_2)$.}
$$
This indicates that
 \bes\label{3.14}
d_S\hat w(x)+\hat I(x)=\hat a+\hat bx,\ \ x\in(\varrho_1,\varrho_2)
 \ees
for some constants $\hat a,\,\hat b$.

In what follows, we aim to determine $\hat a$ and $\hat b$. By a simple observation, $(w,I)$ satisfies
\begin{equation}\nonumber
\begin{cases}
-(d_S w+I)_{xx}=0,\ \ x\in(0,L),\vspace{1mm} \\
(d_S w+I)_{x}=0,\ \ \ \ \,\ x=0,L.
\end{cases}
\end{equation}
Thus, $d_Sw+ I=c_{d_I}$ is a positive constant on $[0,L]$ for any $d_I>0$. Recall that $w,\, I$ are locally uniformly bounded in $(\varrho_1,\varrho_2)$. Hence, as $d_I\to0$, we may assume that
$$d_Sw+ I=c_{d_I}\to\hat c\in[0,\infty)\ \  \mbox{ uniformly on $[0,L]$.}
$$
From \eqref{3.14} it follows that $\hat c=\hat a$ and $\hat b=0$. In addition, our analysis indicates that $w$ and $I$ are uniformly bounded on $[0,L]$. Precisely, it holds that
 \bes\label{3.15}
 w(x),\ \ I(x)\leq C,\ \ \forall x\in[0,L].
 \ees

We now use the equation of $I$, together with the fact of $w,\,I\geq0$ and the definition of $k$, to find that
 \begin{align}
-I_{xx}&=\frac{\beta(x)\left[S-k(x)\right]I}{d_I}\nonumber \\
& =\beta(x)\left[\frac{S-k_{min}}{d_I}+\frac{k_{min}-k(x)}{d_I}\right]I\label{3.16}\\
&\leq \beta(x)wI,\ \ \ \  \ \ x\in(0,L).
\nonumber
\end{align}
Multiplying both sides in \eqref{3.16} by $I$ and integrating over $(0,L)$, we obtain
 $$
 \int_0^L(I_x)^2{\rm d}x\leq\int_0^L\beta wI^2{\rm d}x\leq C
 $$
due to \eqref{3.15}. This and \eqref{3.15} imply that $\|I\|_{W^{1,2}(0,L)}\leq C$. Since $W^{1,2}(0,L)$ is compactly embedded to $C([0,L])$, we can assume that $I\to\hat I$ uniformly on $[0,L]$. By what was proved before, $\hat I=0$ on $[0,\varrho_1]\cup[\varrho_2,L]$, and by \eqref{3.11a} and \eqref{3.14}, on $[\varrho_1,\varrho_2]$, $\hat I$ solves
\begin{equation}\label{3.17}
\begin{cases}
-\hat I_{xx}=\frac{\beta(x)}{d_S}(\hat a-\hat I)\hat I,\ \ &\varrho_1<x<\varrho_2,\vspace{1mm} \\
\hat I=0,\ \ &x=\varrho_1,\,\varrho_2.
\end{cases}
\end{equation}
Because of $\int^{L}_{0}(S(x)+I(x))\,{\rm d}x=N$ and $S\to k_{min}$ uniformly on $[0,L]$ as $d_I\to0$, it is easily seen that
 \bes\label{3.18}
 \int_{\varrho_1}^{\varrho_2}\hat I\,{\rm d}x=N-Lk_{min}>0.
 \ees
Thanks to the Harnack inequality (see Lemma \ref{a2.3}(b)) and \eqref{3.18}, we have from \eqref{3.17} that $\hat I>0$ in $(\varrho_1,\varrho_2)$. By \eqref{3.14} and the fact of $\hat b=0$, clearly $\hat a>0$.

It is well known that given $\hat a>0$, the positive solution of problem \eqref{3.17}, if it exists, must be unique, denoted by $\hat I_{\hat a}$; moreover, if $0<\hat a_1<\hat a_2$, then $\hat I_{\hat a_1}(x)<\hat I_{\hat a_2}(x)$ for all $x\in(\varrho_1,\,\varrho_2)$. With these facts, one can check that the positive constant $\hat a$ is uniquely determined by \eqref{3.18} in an implicit manner. Therefore, all the assertions in case (ii) have been verified. The proof is thus complete.
\end{proof}

\subsection{Proof of Theorem \ref{th2.2}}

We are now in a position to give the proof of Theorem \ref{th2.2}.

\begin{proof}[Proof of Theorem \ref{th2.2}] First of all, one can follow the analysis of Theorem \ref{th2.1},
combined with the result of Theorem \ref{mainthm2} and its proof (see \cite[Theorem 3.2]{LPW1}), to show that
as $d_I\rightarrow0$, any EE $\left(S,I\right)$ of \eqref{SIS-2} satisfies (up to a subsequence of $d_I$) that
$S\to\hat S$ weakly in $W^{1,2}(0,L)$ and uniformly on $[0,L]$, and $I\to\mu$ weakly in the sense of \eqref{wks} for some Radon measure $\mu$ and positive function $\hat S\in W^{1,2}(0,L)$, and
\bes\label{3.aa}
0<\hat S(x)\leq h(x),\ \ \forall x\in[0,L],
\ees
and \eqref{th2.2-3} hold.

For any $\zeta\in W^{1,2}(0,L)$ (and so $\zeta\in C([0,L])$), we use the $S$-equation to obtain
 \begin{align}
d_S\int_0^LS_x\zeta_x{\rm d}x&=\int_0^L[\Lambda-S-\beta(x)SI+\gamma(x)I]\zeta{\rm d}x\nonumber \\
& =\int_0^L[\Lambda-S-\eta(x)I]\zeta{\rm d}x-\int_0^L[\beta(x)S-(\gamma(x)+\eta(x))]I\zeta{\rm d}x\label{3.20}
\end{align}
for all $d_I>0$. In view of \eqref{3.aa} and \eqref{th2.2-3}, we send $d_I\to0$ to infer that
$$
\int_0^L[\beta(x)S-(\gamma(x)+\eta(x)]I\zeta{\rm d}x\to\int_{[0,L]}\beta(x)[\hat S-h(x)]\zeta\mu({\rm d}x)=0.
$$
Thus, by letting $d_I\to0$, it follows from \eqref{3.20} that
\bes\label{3.20a}
d_S\int_0^L\hat S_x\zeta_x{\rm d}x=\int_0^L(\Lambda-\hat S)\zeta{\rm d}x-\int_0^L \eta(x)\zeta\mu({\rm d}x),\ \ \forall \zeta\in W^{1,2}(0,L).
\ees
Together with \eqref{th2.2-3}, this means that $\hat S\in W^{1,2}(0,L)$ is a weak solution of \eqref{th2.2-1}.

In what follows, for a general positive H\"{o}lder continuous function $h$, we will prove three claims:

\vskip6pt
{\bf Claim 1.} If the minimum of $h$ is attained at $x=0$ (resp. at $x=L$), then $\hat S$ must touch $h$ at this point; that is, $\hat S(0)=h(0)=h_{min}$ (resp. $\hat S(L)=h(L)=h_{min}$).

We only handle the case that $h_{min}$ is attained at $x=0$, and the other case can be treated similarly.
Since $\hat S\leq h$ on $[0,L]$, we suppose that $\hat S(0)<h(0)$ and so $\hat S(x)<h(x)$ on $[0,\epsilon_0]$ for some small $\epsilon_0>0$. Thus, from \eqref{th2.2-1}, we have $-d_S\hat S_{xx}=\Lambda-\hat S,\ \forall x\in(0,\epsilon_0]$.
A simple analysis shows that
\bes\label{3.21}
\hat S(x)=c_1e^{d_S^{-1/2}x}+c_2e^{-d_S^{-1/2}x}+\Lambda,\,\  x\in(0,\epsilon_0]
\ees
for some constants $c_1,\,c_2$. On the other hand,
using the $S$-equation, we integrate on $[0,x]$ to deduce
\bes\label{3.21-aa}
-S_x(x)=\frac{1}{d_S}\int_0^x[\Lambda-S(y)-\beta(y)S(y)I(y)+\gamma(y)I(y)]{\rm d}y,\ \ x\in[0,\epsilon_0].
\ees
From the proof of \cite[Theorem 3.2]{LPW1}, we know that
\bes\label{3.21-ab}
\int_0^L S(x)I(x){\rm d}x\leq C,\ \ \int_0^L I(x){\rm d}x\leq C,\ \mbox{and}\  S(x)\leq C,\ \ \ \forall x\in[0,L],
\ees
for some positive constant $C$, independent of $d_I>0$.

In the sequel, the constant $C$ allows to vary from line to line but does not depend on $d_I>0$. It immediately follows from \eqref{3.21-aa} that $S_x$ is uniformly bounded on $[0,\epsilon_0]$, independent of $d_I>0$.
Note that $\mu([0,\epsilon_0])=0$ due to \eqref{th2.2-3}, and $I\to\mu$ weakly in the sense of \eqref{wks}.
Given any small $\epsilon>0$, we can find a small $\rho>0$ so that for all $0<d_I\leq\rho$,
$$
\int_0^{\epsilon_0}I(x){\rm d}x\leq\epsilon+\int_{[0,\epsilon_0]}\mu({\rm d}x)=\epsilon.
$$
Now, for any $x_1,\,x_2\in[0,\epsilon_0]$ satisfying $|x_1-x_2|<\epsilon$, we have
\begin{align}
\big|S_x(x_1)-S_x(x_2)\big|&=\frac{1}{d_S}\Big|\int_{x_1}^{x_2}[\Lambda-S(y)-\beta(y)S(y)I(y)+\gamma(y)I(y)]{\rm d}y\Big|\nonumber \\
&\leq C|x_1-x_2|+C\int_{x_1}^{x_2}I(y){\rm d}y\nonumber \\
&\leq C|x_1-x_2|+C\int_{0}^{\epsilon_0}I(y){\rm d}y\leq C\epsilon\nonumber
\end{align}
provided that $0<d_I\leq\rho$. This shows that $S_x$ is equi-continuous on $[0,\epsilon_0]$ once $0<d_I\leq\rho$.

Hence, we can apply the well-known Ascoli-Arzel\`{a} theorem, up to a further subsequence of $d_I$, to conclude that $S_x$ is uniformly convergent on $[0,\epsilon_0]$ as $d_I\to0$. As
$$
S(x)-S(0)=\int_0^x S_x(y){\rm d}y,\ \ \ S\to\hat S\ \ \mbox{uniformly on}\ [0,\epsilon_0],
$$
it is easily seen that $S\to\hat S$ in $C^1([0,\epsilon_0])$. Thus, $\hat S_x(0)=0$, and in turn we get from \eqref{3.21} that $c_1=c_2$.
Because of $\hat S\leq h$ on $[0,L]$ and the condition \eqref{cond}, we have $c_1=c_2<0$, and so
$$\hat S_x(x)=c_1[e^{d_S^{-1/2}x}-e^{-d_S^{-1/2}x}]<0,\ \ \, \forall x\in(0,\epsilon_0].$$
This means that $\hat S$ is decreasing on $[0,\epsilon_0]$.

By virtue of $h(0)\leq h(x)$ for all $x\in[0,L]$ and \eqref{th2.2-3}, one can extend the above analysis to assert that $\hat S$ is decreasing on $[0,L]$ and so
$\hat S<h$ on $[0,L]$. This clearly gives $\mu([0,L])=0$, a contradiction with $\mu([0,L])>0$ due to \eqref{th2.2-3} again. Hence, we must have $\hat S(0)=h(0)=h_{min}$.

\vskip6pt {\bf Claim 2.} If $\hat S$ attains its local minimum at some $x_0\in(0,L)$, then $\hat S$ must touch $h$ at this point; that is,
$\hat S(x_0)=h(x_0)$.

Suppose that $\hat S(x_0)<h(x_0)$ due to $\hat S\leq h$. Thus, there is a small $\epsilon_0>0$ such that $\hat S(x)<h(x)$ for all $x\in[x_0-\epsilon_0,x_0+\epsilon_0]\subset(0,L)$. By \eqref{th2.2-3}, $\mu([x_0-\epsilon_0,x_0+\epsilon_0])=0$ and so
$$-d_S\hat S_{xx}=\Lambda-\hat S\ \ \mbox{ on $[x_0-\epsilon_0,x_0+\epsilon_0]$.}
$$
As before, $\hat S$ takes the form of
\eqref{3.21} on $[x_0-\epsilon_0,x_0+\epsilon_0]$ for some constants $c_1,\,c_2$. Obviously, $\hat S_x(x_0)=0$, which leads to $c_2=c_1e^{2d_S^{-1/2}x_0}$, and so $c_1<0$. Thus, it holds that
\bes\label{3.22}
\hat S(x)=c_1[e^{d_S^{-1/2}x}+e^{d_S^{-1/2}(2x_0-x)}]+\Lambda,\,\ \ x\in[x_0-\epsilon_0,x_0+\epsilon_0]
\ees
for some constant $c_1<0$.
In view of \eqref{3.22}, basic computation gives that $\hat S$ is increasing on $[x_0-\epsilon_0,x_0]$ while is decreasing on $[x_0,x_0+\epsilon_0]$.
This implies that $x_0$ is a local maximum of $\hat S$, a contradiction with our assumption. As a result, $\hat S$ must touch $h$ at $x=x_0$.

\vskip6pt
{\bf Claim 3.} If the minimum of $h$ is attained at some point $y_0\in(0,L)$, then $\hat S$ must touch $h$ at this point; that is, $\hat S(y_0)=h(y_0)=h_{min}$.

Suppose that $\hat S(y_0)<h(y_0)=h_{min}$. There are two possible cases to happen in the interval $[0,y_0)$:
Case 1. $\hat S$ never touches $h$ in $[0,y_0)$, that is, $\hat S<h$ in $[0,y_0)$; Case 2. $\hat S$ touches $h$ somewhere in $[0,y_0)$.

When Case 1 occurs, by \eqref{th2.2-3}, we know that $\hat S$ must touch $h$ in $(y_0,L]$. Let $y_1$ be the first point (from the left side) at which $\hat S$ touches $h$. That is, $y_1\in(y_0,L]$, and
$$\hat S(x)<h(x), \  \ \forall x\in(y_0,y_1),\ \  \ \hat S(y_1)=h(y_1)\geq h_{min}.
$$
On the other hand, since $\hat S<h$ in $[0,y_0)$, we can follow the analysis used in {\bf Claim 1} to show that $\hat S$ is decreasing on $[0,y_1]$. This is an obvious contradiction with $\hat S(y_0)<h_{min}\leq \hat S(y_1)$.

When Case 2 occurs, we denote by $y_2\in[0,y_0)$ the first point from the right side such that $\hat S$ touches $h$ in $[0,y_0)$. That is,
$$\hat S(x)<h(x),\ \ \forall x\in(y_2,y_0),\ \ \ \hat S(y_2)=h(y_2)\geq h_{min}.
$$ If $\hat S$ does not touch $h$ in $(y_0,L]$. By a similar argument to the proof of {\bf Claim 1} and appealing to the fact of $S_x(L)=0$, one sees that $\hat S$ is increasing in $(y_2,L]$, leading to $\hat S(y_2)<\hat S(y_0)$, which contradicts with $\hat S(y_2)\geq h_{min}>\hat S(y_0)$. Hence, it is necessary that $\hat S$ touches $h$ in $(y_0,L]$. Let $y_3$  be the first point where $\hat S$ touches $h$ in $(y_0,L]$. Thus, $\hat S(x)<h(x)$ for all $x\in(y_0,y_3)$ and $\hat S(y_3)=h(y_3)\geq h_{min}$. Therefore, $\hat S(x)<h(x)$ in the interval $(y_2,y_3)$, $\hat S(y_0)<h(y_0)=h_{min}$ and $\hat S(y_2),\,\hat S(y_3)\geq h_{min}$. This implies that on $[y_2,y_3]$, $\hat S$ must attain its minimum at some $y_4\in(y_2,y_3)$. By {\bf Claim 2}, we can conclude that $\hat S(y_4)=h(y_4)$, a contradiction again. So far, we have verified {\bf Claim 3}.

A similar reasoning as that of proving {\bf Claim 3} yields $\hat S\geq h_{min}$ on $[0,L]$. Thus \eqref{th2.2-2} holds. Thanks to  {\bf Claim 1} and {\bf Claim 3}, \eqref{th2.2-2a} is true. It is also apparent that {\bf Claim 2} implies \eqref{th2.2-2b}. The proof is now complete.
\end{proof}



\subsection{Proof of Theorem \ref{th2.3}}
This subsection is devoted to the proof of Theorem \ref{th2.3}. We begin with some lemmas as follows.

\begin{lem}\label{l2.3} Assume that $h\in C^2([0,L])$ and $h_{x}$ is non-decreasing in some neighborhood of $\varrho_0\in\Theta_h$. Let $\hat S$ and $\mu$ be given as in Theorem \ref{th2.2}. Then there exists a small $\epsilon_0>0$ such that
$$
\hat S(x)=h(x),\ \ \ \forall x\in(\varrho_0-\epsilon_0,\varrho_0+\epsilon_0)\cap(0,L)
$$
and
$$
\mu(\{x\})=\frac{\Lambda-h+d_Sh_{xx}}{\eta(x)},\ \ \ \mbox{ a.e. for}\ \, x\in(\varrho_0-\epsilon_0,\varrho_0+\epsilon_0)\cap(0,L).
$$

\end{lem}

\begin{proof} By Theorem \ref{th2.2}, we know that $\varrho_0\in\{x\in[0,L]:\ \hat S(x)=h(x)\}$. In the sequel, we only consider the case of $\varrho_0\in(0,L)$, and the case of $\varrho_0=0$ or $L$ can be treated similarly. There are three possibilities we have to distinguish:

(1)\ $\varrho_0$ is an isolated point in the set $\{x\in[0,L]:\ \hat S(x)=h(x)\}$;

(2)\ $\varrho_0$ is an accumulation point in  $\{x\in[0,L]:\ \hat S(x)=h(x)\}$;

(3)\ there is a small $\epsilon_0>0$ such that $(\varrho_0-\epsilon_0,\varrho_0+\epsilon_0)\subset\{x\in[0,L]:\ \hat S(x)=h(x)\}$.

In what follows, we will exclude (1) and (2). If (1) happens, then
$$\hat S(\varrho_0)=h(\varrho_0)=h_{min}\ \ \mbox{and}\ \ \hat S<h\ \ \mbox{ in $(\varrho_0-\epsilon_1,\varrho_0+\epsilon_1)\setminus\{\varrho_0\}$}$$
for some small $\epsilon_1>0$.

Note that $\mu([0,L])<\infty$. In view of this fact, one can apply the interior regularity theory for elliptic equations to \eqref{th2.2-1} and assert that $\hat S\in C^1(0,L)$. Clearly, $h_x(\varrho_0)=0$. Since $\hat S(\varrho_0)=h(\varrho_0)=h_{min}$ and $\hat S\geq h_{min}$ due to \eqref{th2.2-2}, we infer that $\hat S_x(\varrho_0)=0$.

On the other hand, by \eqref{th2.2-1}, $\hat S$ satisfies
\bes\label{3.26}
-d_S\hat S_{xx}=\Lambda-\hat S\ \ \ \mbox{ in}\ (\varrho_0-\epsilon_1,\varrho_0+\epsilon_1)\setminus\{\varrho_0\}.
\ees
By using $\hat S_x(\varrho_0)=0$ and \eqref{3.26}, one can easily see that $\hat S$ is increasing in $(\varrho_0-\epsilon_1,\varrho_0)$ while $\hat S$ is decreasing in $(\varrho_0,\varrho_0+\epsilon_1)$. This implies that $\hat S<h_{min}$ in $(\varrho_0-\epsilon_1,\varrho_0+\epsilon_1)\setminus\{\varrho_0\}$, contradicting against \eqref{th2.2-2}. Thus, (1) is impossible.

If (2) happens, without loss of generality, we can find two points, say $z_1,\,z_2$ with $\varrho_0<z_1<z_2<\varrho_0+\epsilon_2$ for some small $\epsilon_2>0$ such that
\bes\label{3.27}
\mbox{$\hat S(z_1)=h(z_1)$,\ \ $\hat S(z_2)=h(z_2)$\ \ and\ \ $\hat S<h$\ \ in\ $(z_1,z_2)$.}
\ees
By taking $\epsilon_2$ to be smaller if necessary, we may assume that $h_x(z_1)\leq h_x(z_2)$ due to the monotonicity of $h_{x}$. Then, $\hat S$ solves \eqref{3.26} in $(z_1,z_2)$. By means of \eqref{3.27}, we have
$$
\hat S_x(z_1)\leq h_x(z_1),\ \ \ \hat S_x(z_2)\geq h_x(z_2),
$$
leading to $\hat S_x(z_1)\leq\hat S_x(z_2)$. However, it follows from \eqref{3.26} that $\hat S_{xx}<0$ in $(z_1,z_2)$, which gives
$\hat S_x(z_1)>\hat S_x(z_2)$, a contradiction. Hence, the possibility (2) has been ruled out.

The above argument shows that (3) must hold. Now, since $\hat S=h$ on $[\varrho_0-\epsilon_0,\varrho_0+\epsilon_0]$, we can multiply  both sides of \eqref{th2.2-3} by any function $\zeta\in C^2([0,L])$ with compact support on $[\varrho_0-\epsilon_0,\varrho_0+\epsilon_0]$ and integrate to conclude that
$$
d_S h_{xx}+\Lambda-h-\eta(x)\mu(\{x\})=0,\ \ \mbox{ a.e. for}\ \  x\in(\varrho_0-\epsilon_0,\varrho_0+\epsilon_0),
$$
which yields the expression of $\mu(\{x\})$.
\end{proof}

\begin{lem}\label{l2.4} Assume that $h\in C^2([0,L])$, $h_x$ is non-decreasing on $[0,L]$, and $\Theta_h=\{\tau_0\}$ for some $\tau_0\in(0,L)$. Then there exist two numbers $\tau_1,\,\tau_2$ with $0<\tau_1<\tau_0<\tau_2<L$ such that
 \bes\label{3.28}
 \hat S(x)=h(x),\ \ \ \ \forall x\in[\tau_1,\tau_2],
 \ees
and on $[0,\tau_1)\cup(\tau_2,L]$, $\hat S$ satisfies
\begin{equation}\label{3.29}
\begin{cases}
 -d_S\hat S_{xx}(x)=\Lambda-\hat S,\ \ \ \  \ x\in(0,\tau_1)\cup(\tau_2,L),\vspace{1mm} \\
\hat S_x(0)=\hat S_x(L)=0,\ \ \vspace{1mm} \\
\hat S(\tau_1)=h(\tau_1),\ \ \hat S(\tau_2)=h(\tau_2),
\end{cases}
\end{equation}
and $\mu$ satisfies
 \bes\label{3.28-a}
 \mu(\{x\})=\frac{\Lambda-h+d_Sh_{xx}}{\eta(x)},\ \ \mbox{ a.e. for}\ \, x\in(\tau_1,\tau_2),
 \ees
\bes\label{3.28-b}
 \mu(\{x\})=0,\ \ \ \ \forall x\in[0,\tau_1)\cup(\tau_2,L].
 \ees

\end{lem}

\begin{proof} Let us denote
$$
\tau_1=\inf\{\tau\in[0,\tau_0):\ \ \hat S(x)=h(x),\ \forall x\in[\tau,\tau_0]\},$$
$$
\tau_2=\sup\{\tau\in(\tau_0,L]:\ \ \hat S(x)=h(x),\ \forall x\in[\tau_0,\tau]\}.
$$
Lemma \ref{l2.3} implies that $\tau_1$ and $\tau_2$ are well defined, and $0\leq\tau_1<\tau_0$ and $\tau_0<\tau_2\leq L$. In addition, \eqref{3.28} and \eqref{3.28-a} hold.

In light of the monotonicity of $h_{x}$ on $[0,L]$, it is easily seen from the proof of Lemma \ref{l2.3} that if $\tau_1>0$, then
$\hat S$ can not touch $h$ in $(0,\tau_1)$ and in turn $\mu([0,\tau_1))=0$; similarly, if $\tau_2<L$, $\hat S$ can not touch $h$ in  $(\tau_2,L)$ and so $\mu((\tau_2,L])=0$.

If $\tau_1>0$ and $\tau_2<L$, we can use the analysis as in the proof of Claim 1 of Theorem \ref{th2.2} to conclude that $\hat S_x(0)=\hat S_x(L)=0$. As $\mu([0,\tau_1)\cup(\tau_2,L])=0$, by \eqref{th2.2-1} and the continuity of $\hat S$, a standard compactness argument of elliptic equations yields that $\hat S$ solves \eqref{3.29} in the classical sense. Clearly, the solution of \eqref{3.29} is unique.

It remains to prove $\tau_1>0$ and $\tau_2<L$. Note that the monotonicity of $h_{x}$, $\Theta_h=\{\tau_0\}$ and $h_x(\tau_0)=0$ ensure
$h_{x}(0)<0$ and $h_x(L)>0$. Suppose that $\tau_1=0$, and so \eqref{3.28} holds on $[0,\tau_2]$.
Now, given $\tau\in(0,\tau_0]$, integrating the $S$-equation over $[0,\tau]$ and using \eqref{3.28}, we infer that
\begin{align}
-d_SS_x(\tau^-)&=\int_0^{\tau}[\Lambda-S(y)-\beta(y)S(y)I(y)+\gamma(y)I(y)]{\rm d}y\nonumber \\
&=\int_0^{\tau}[\Lambda-S(y)-\eta(y)I(y)]{\rm d}y+\int_0^{\tau}[\gamma(y)+\eta(y)-\beta(y)S(y)]I(y){\rm d}y\nonumber \\
&\rightarrow\int_{[0,\tau]}[\Lambda-h(y)-\eta(y)\mu]({\rm d}y)
=\int_0^{\tau}[-d_Sh_{xx}(y)]{\rm d}y\nonumber\\
&=-d_Sh_x(\tau)+d_Sh_x(0),\ \ \ \mbox{as}\ \, d_I\to0.
\nonumber
\end{align}
That is, for any $\tau\in(0,\tau_0]$, it holds that
\bes\nonumber
S_x(\tau^-)\to h_x(\tau)-h_x(0),\ \ \ \mbox{as}\ \, d_I\to0.
\ees
Since $h_{x}$ is non-decreasing on $[0,\tau_0]$ and $h_{x}(0)<0$, there exists a small $\epsilon_0>0$ such that for all $x\in[\tau_0-\epsilon_0,\tau_0]$,
\bes\nonumber
S_x(x^-)\geq \frac{1}{2}[h_x(\tau_0)-h_x(0)]=-\frac{1}{2}h_x(0)>0
\ees
for all small $d_I>0$. This implies that $S$ is increasing on $[\tau_0-\epsilon_0,\tau_0]$ for all such small $d_I>0$.
In view of $S\to h$ uniformly on $[\tau_0-\epsilon_0,\tau_0]$ as $d_I\to0$, $h$ must be non-decreasing on $[\tau_0-\epsilon_0,\tau_0]$, which is a contradiction against our assumption.
Hence, $\tau_1>0$. Similarly, we have $\tau_2<L$ by using $h_x(L)>0$. As a consequence, we deduce \eqref{3.28-b}. The proof is complete.
\end{proof}

Similar to the argument of Lemma \ref{l2.3}, we can conclude the following result.

\begin{lem}\label{l2.3a} Assume that $h\in C^2([0,L])$, $[\varrho_1,\varrho_2]\subset\Theta_h$ and $h_{x}$ is non-decreasing in some neighborhood of $\varrho_1,\,\varrho_2$. Let $\hat S$ and $\mu$ be given as in Theorem \ref{th2.2}. Then there exists a small $\epsilon_0>0$ such that
$$
\hat S(x)=h(x),\ \ \ \forall x\in(\varrho_1-\epsilon_0,\varrho_2+\epsilon_0)\cap(0,L)
$$
and
$$
\mu(\{x\})=\frac{\Lambda-h+d_Sh_{xx}}{\eta(x)},\ \ \ \mbox{ a.e. for}\ \, x\in(\varrho_1-\epsilon_0,\varrho_2+\epsilon_0)\cap(0,L).
$$

\end{lem}

Based upon Lemma \ref{l2.3a}, we can deduce the following result.

\begin{lem}\label{l2.4a} Assume that $h\in C^2([0,L])$, $\Theta_h=[\varrho_1,\varrho_2]$ and $h_{x}$ is non-decreasing on $[0,\varrho_1]\cup[\varrho_2,L]$. Let $\hat S$ and $\mu$ be given as in Theorem \ref{th2.2}. Then there exist two numbers $\tau_1,\,\tau_2$ with $0<\tau_1<\varrho_1<\varrho_2<\tau_2<L$ such that all the assertions in Lemma \ref{l2.4} hold.
\end{lem}

\vskip10pt With the aid of Lemmas \ref{l2.3}-\ref{l2.4a}, we are now in a position to prove Theorem \ref{th2.3}.

\begin{proof}[Proof of Theorem \ref{th2.3}] We first prove (i). We proceed indirectly and suppose that $\hat S\not\equiv h$ on $[0,L]$. Since $\hat S$ touches $h$ at least at the highest-risk point due to Theorem \ref{th2.2}, we can find an interval, denoted by $[\ell_1,\ell_2]\subset[0,L]$, such that $\hat S<h$ in $(\ell_1,\ell_2)$ and at the boundary point $x=\ell_i$ for $i=1,\,2$, either
$\hat S$ touches $h$ (and so $\hat S(\ell_i)=h(\ell_i)$) or $\hat S(\ell_i)<h(\ell_i)$. In the latter case, it is necessary that $\ell_i=0$ or $L$, and the analysis to deduce Claim 1 in the proof of Theorem \ref{th2.2} shows that $\hat S_x(\ell_i)=0$. In any case, clearly $\hat S$ satisfies
\begin{equation}\label{th2.3-aa}
\begin{cases}
-d_S\hat S_{xx}=\Lambda-\hat S,\ \ \ \ x\in(\ell_1,\ell_2),\vspace{1mm} \\
\hat S(\ell_i)=h(\ell_i)\ \ \mbox{or}\ \ \hat S_x(\ell_i)=0,\ \ i=1,2.
\end{cases}
\end{equation}
Thus, by our assumption, $h$ is a sub-solution to problem \eqref{th2.3-aa}, and $\max\{\Lambda,\ \max_{x\in[0,L]}h(x)\}$ is a super-solution to \eqref{th2.3-aa}. The well-known technique of sub-supersolution iteration, combined with the uniqueness of solutions to problem \eqref{th2.3-aa}, allows us to conclude that $\hat S\geq h$ on $[\ell_1,\ell_2]$, which leads to a contradiction. Hence, \eqref{th2.3-a} holds, and \eqref{th2.3-b} follows from \eqref{th2.2-1} by using a test-function argument similarly as before.
Therefore, (i) is proved.

We next prove (ii). First of all, let us consider the case of $\tau_0\in(0,L)$. In this case,
the assertions \eqref{th2.3-c}-\eqref{th2.3-f} follow from Lemma \ref{l2.4}, and it remains to show that $\tau_1,\,\tau_2$ are uniquely determined by \eqref{th2.3-g}. As $\hat S<h$ in $[0,\tau_1)$, we have
$$
\hat S(x)=c_1[e^{d_S^{-1/2}x}+e^{-d_S^{-1/2}x}]+\Lambda,\ \, \ \forall x\in[0,\tau_1]
$$
for some $c_1<0$. It then follows from $\hat S(\tau_1)=h(\tau_1)$ that
$$
\hat S(x)=-\frac{\Lambda-h(\tau_1)}{e^{{d_S^{-1/2}\tau_1}}+e^{-{d_S^{-1/2}\tau_1}}}(e^{d_S^{-1/2}x}+e^{-d_S^{-1/2}x})+\Lambda,\ \ \ \forall x\in[0,\tau_1].
$$
Note that $\hat S$ is convex while $h$ is concave in the interval $[0,\tau_1)$, and moreover, $\hat S\in C^1([0,L])$ as shown before. Hence, $\hat S$ must be tangent to $h$ at $x=\tau_1$, which in turn implies that $\tau_1$ is the unique solution to $\hat S_x(\tau_1)=h_x(\tau_1)$. Thus, $\tau_1$ is uniquely determined by the following equation:
$$
\frac{e^{{d_S^{-1/2}\tau_1}}-e^{-{d_S^{-1/2}\tau_1}}}{e^{{d_S^{-1/2}\tau_1}}+e^{-{d_S^{-1/2}\tau_1}}}=-\frac{d_S^{1/2}h_x(\tau_1)}{\Lambda-h(\tau_1)}.
$$
Similarly, $\tau_2$ is uniquely determined by the second equation of \eqref{th2.3-g}. The assertions in (ii)-(a) have been verified.

We now consider the case of $\tau_0=L$. In view of our assumption, clearly $h_x(0)<0$, $h_x(L)\leq0$, and $\hat S(L)=h(L)$.

Assume that $\frac{e^{2Ld_S^{-1/2}}-1}{e^{2Ld_S^{-1/2}}+1}>-\frac{d_S^{1/2}h_x(L)}{\Lambda-h(L)}$. In order to deduce the desired conclusion in (ii)-(b1), one can follow the analysis of Lemmas \ref{l2.3} and \ref{l2.4}. By checking the analysis there, one just needs to show that $\tau_1$ defined in the assertion (ii)-(a) satisfies $\tau_1>0$. It turns out that this amounts to rule out the situation that $\hat S<h$ in $[0,L)$. Suppose that $\hat S<h$ in $[0,L)$.
Then, arguing as before, we see that $\hat S$ satisfies $-d_S\hat S_{xx}=\Lambda-\hat S$ in $(0,L)$ and $\hat S_x(0)=0$.
Solving this problem, we get $\hat S(x)=c_1[e^{{d_S^{-1/2}x}}+e^{-{d_S^{-1/2}x}}]+\Lambda$ for some $c_1<0$. It then follows from $\hat S(L)=h(L)$ that
$$
c_1=-\frac{\Lambda-h(L)}{e^{{d_S^{-1/2}L}}+e^{-{d_S^{-1/2}L}}}.
$$
Thus, we get
$$
\hat S_x(L)=-d_S^{-1/2}(\Lambda-h(L))\frac{e^{2Ld_S^{-1/2}}-1}{e^{2Ld_S^{-1/2}}+1}.
$$
By means of $\hat S<h$ in $[0,L)$ and $\hat S(L)=h(L)$, it is necessary that $\hat S_x(L)\geq h_x(L)$, which leads to
$$\frac{e^{2Ld_S^{-1/2}}-1}{e^{2Ld_S^{-1/2}}+1}\leq-\frac{d_S^{1/2}h_x(L)}{\Lambda-h(L)},$$
contradicting with our assumption. Therefore, $\tau_1>0$ must hold,  and (ii)-(b1) is proved.

Assume that $\frac{e^{2Ld_S^{-1/2}}-1}{e^{2Ld_S^{-1/2}}+1}\leq-\frac{d_S^{1/2}h_x(L)}{\Lambda-h(L)}$.  We first show that $\tau_1>0$ is impossible. On the contrary, we suppose that $\tau_1>0$, and by the above analysis, $\tau_1$ must solve the first equation of \eqref{th2.3-g}. Let us consider the following auxiliary problem:
$$f(\tau)=\frac{e^{2\tau d_S^{-1/2}}-1}{e^{2\tau d_S^{-1/2}}+1}+\frac{d_S^{1/2}h_x(\tau)}{\Lambda-h(\tau)},\ \ \ \ \tau\in[0,L].$$
Since $h_x(\tau)$ is non-decreasing, $h_x(\tau)\leq0$ on $[0,L]$, $h(\tau)$  is non-increasing and $h(\tau)>\Lambda$ on $[0,L]$, it is easy to check that $\frac{d_S^{1/2}h_x(\tau)}{\Lambda-h(\tau)}$ is non-decreasing  on $[0,L]$. Clearly, $\frac{e^{2\tau d_S^{-1/2}}-1}{e^{2\tau d_S^{-1/2}}+1}$ is increasing on $[0,L]$. Therefore, $f(\tau)$ is increasing on $[0,L]$. Observe that $f(L)=\frac{e^{2Ld_S^{-1/2}}-1}{e^{2Ld_S^{-1/2}}+1}+\frac{d_S^{1/2}h_x(L)}{\Lambda-h(L)}\leq0$ due to our assumption. This implies that the first equation of \eqref{th2.3-g} has no solution with respect to $\tau_1$ in $[0,L)$, arriving at a contradiction. Hence, $\hat S<h$ in $[0,L)$ and $\mu([0,L))=0$, and so $\hat S$ solves \eqref{3.33-b}. It remains to prove
\eqref{th2.3-h}. Indeed, by integrating the sum of \eqref{SIS-2}, we obtain
$$
\Lambda L-\int_0^L S(x){\rm d}x=\int_0^L \eta(x)I(x){\rm d}x,\ \ \forall d_I>0.
$$
Letting $d_I\to0$ yields
$$
\Lambda L-\int_0^L\hat S(x){\rm d}x=\int_{[0,L]} \eta(x)\mu({\rm d}x)=\eta(L)\mu(\{L\}).
$$
Here we used the fact of $\mu([0,L))=0$. This gives \eqref{th2.3-h}, and thus the assertions in (ii)-(b2) hold true.

The case of $\tau_0=0$ can be treated similarly as above. In view of Lemma \ref{l2.4a} and the analysis above, the assertions in (iii) follow immediately. The proof is completed.
\end{proof}

\section{Discussions and numerical simulations}
\setcounter{equation}{0}

In recent years, many reaction-diffusion models have been proposed to investigate the transmission dynamics of infectious diseases in a heterogeneous environment. For example, models associated with \eqref{tSIS} have been studied in \cite{ABLouN2008,DP,DW,Gao,LB,LLT,LPW,LPW1,MWW1,MWW2,Peng,PL,Peng-Yi,PZ,SB,TL,WZ}. When the random diffusion is not present, such kind of models have been explored in \cite{Allen1,Allen3,GD,GR,Li-Shuai,LP,VK,WCT,WN} and the references therein. One may also refer to \cite{CL,GKLZ,HL,LLL,LPX,LXZ,LZZ,SLX,WW1,WW2,ZW1,ZW2} for relevant studies on the effect of random diffusion on the dynamics of infectious diseases.

In this paper, we have investigated the steady state solution (namely, EE) of the SIS epidemic reaction-diffusion models \eqref{SIS} and \eqref{SIS-2}, in which the disease transmission is governed by the well-known mass action infection mechanism, due to Kermack and McKendrick \cite{KM1}. In model \eqref{SIS}, the total population number of the susceptible and infected populations is a constant, while in model \eqref{SIS-2}, the total population number is varying, which results from the inclusion of the recruitment for the susceptible population and the death of the infected population. Our purpose is to determine the spatial profile of EE as the movement rate $d_I$ of the infected individuals tends to zero. Such kind of information may be useful for decision-makers to predict the pattern of disease occurrence and henceforth to develop effective disease control strategies.

The previous works \cite{LPW,WJL} derived partial results regarding the spatial profile of EE for \eqref{SIS} and \eqref{SIS-2} as $d_I\to0$; however, a precise characterization for the distribution of  susceptible and infected populations is lacking. In the present work, we have provided a comprehensive understanding on this issue. Below we shall summarize the main theoretical findings of this paper, which will also be supported or complemented by our numerical simulation results.

\subsection{Profile of EE of model \eqref{SIS} as $d_I\to0$.} As pointed out before, when the risk function $k(x)=\frac{\gamma(x)}{\beta(x)}$ is a constant on the entire habitat $[0,L]$, then $(k,\frac{N}{L}-k)$ is the unique EE of \eqref{SIS} provided that $k<\frac{N}{L}$, while $(\frac{N}{L},0)$ is the unique disease-free equilibrium of \eqref{SIS} provided that $k\geq \frac{N}{L}$. Indeed, in such a trivial case, one can follow the same analysis as in \cite[Theorem 4.1]{DW} to conclude that $(k,\frac{N}{L}-k)$ is a global attractor of \eqref{tSIS} if $k<\frac{N}{L}$ and $(\frac{N}{L},0)$ is a global attractor of \eqref{model} if $k\geq \frac{N}{L}$. Thus, unless otherwise specified, we always assume below that the risk function $k(x)=\frac{\gamma(x)}{\beta(x)}$ is {\it non-constant} on $[0,L]$.

According to Theorem \ref{th2.1}, for model \eqref{SIS}, one finds that the susceptible population $S$ converges to the positive constant $k_{min}$ as $d_I\to0$, which means that the susceptible will always distribute homogeneously on the entire habitat once the movement of the infected individuals is restricted to be sufficiently small. Nevertheless, the profile of the infected population $I$ as $d_I\to0$ crucially depends on the distribution behavior of the highest-risk set $\Theta_k$ of the risk function $k(x)$. More precisely, concerning the profile of $I$ for model \eqref{SIS}, we have the following findings.

\vskip6pt
(i)\ If $\Theta_k$ consists of a single point, then $I$ must concentrate only at such a highest-risk point.

\vskip6pt
(ii)\ If $\Theta_{k}$ contains only multiple isolated points, it follows from Remark \ref{re2.2} that $I$ will also concentrate at least at one of those highest-risk points, and the disease will vanish elsewhere. As shown in Figure \ref{fig1}(a)-(b)-(c) for three typical cases, our simulation results suggest that $I$ should concentrate at all such highest-risk points, though the population number of $I$ at each such highest-risk point may vary, depending on the functions $\beta,\, \gamma$.
\begin{figure}[h]
\centering
\includegraphics[width=5.3cm]{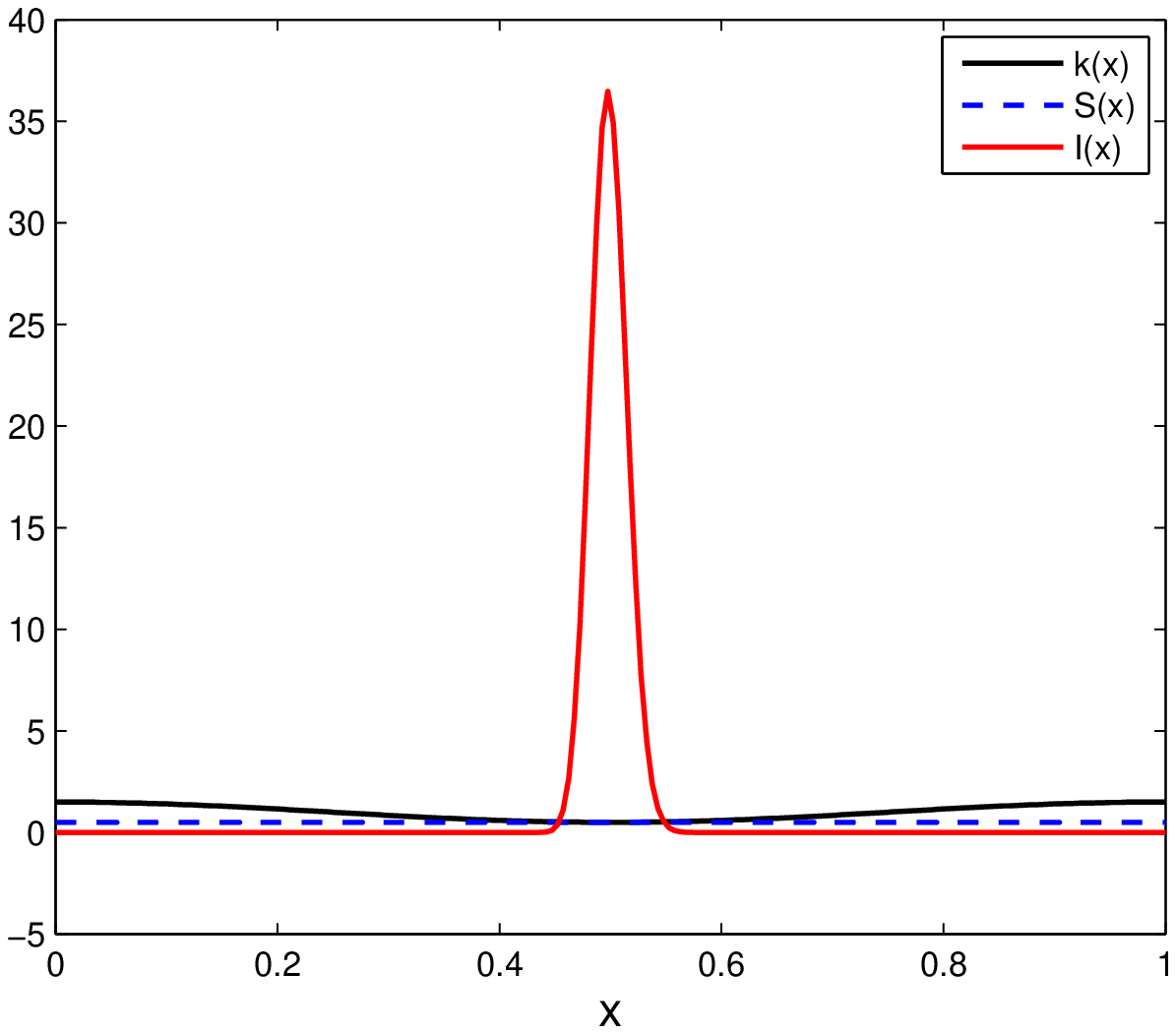}
\includegraphics[width=5.3cm]{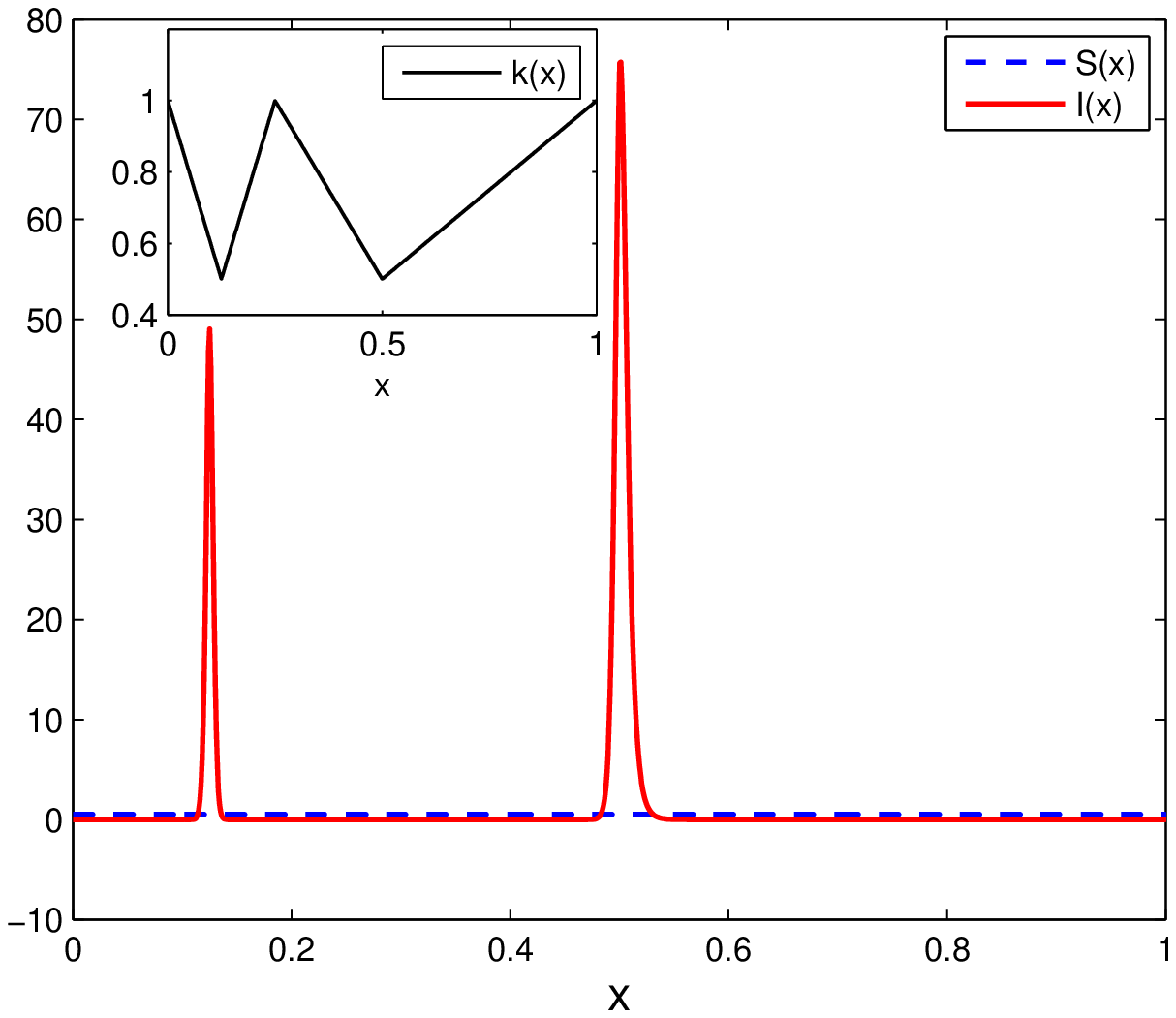}
\includegraphics[width=5.3cm]{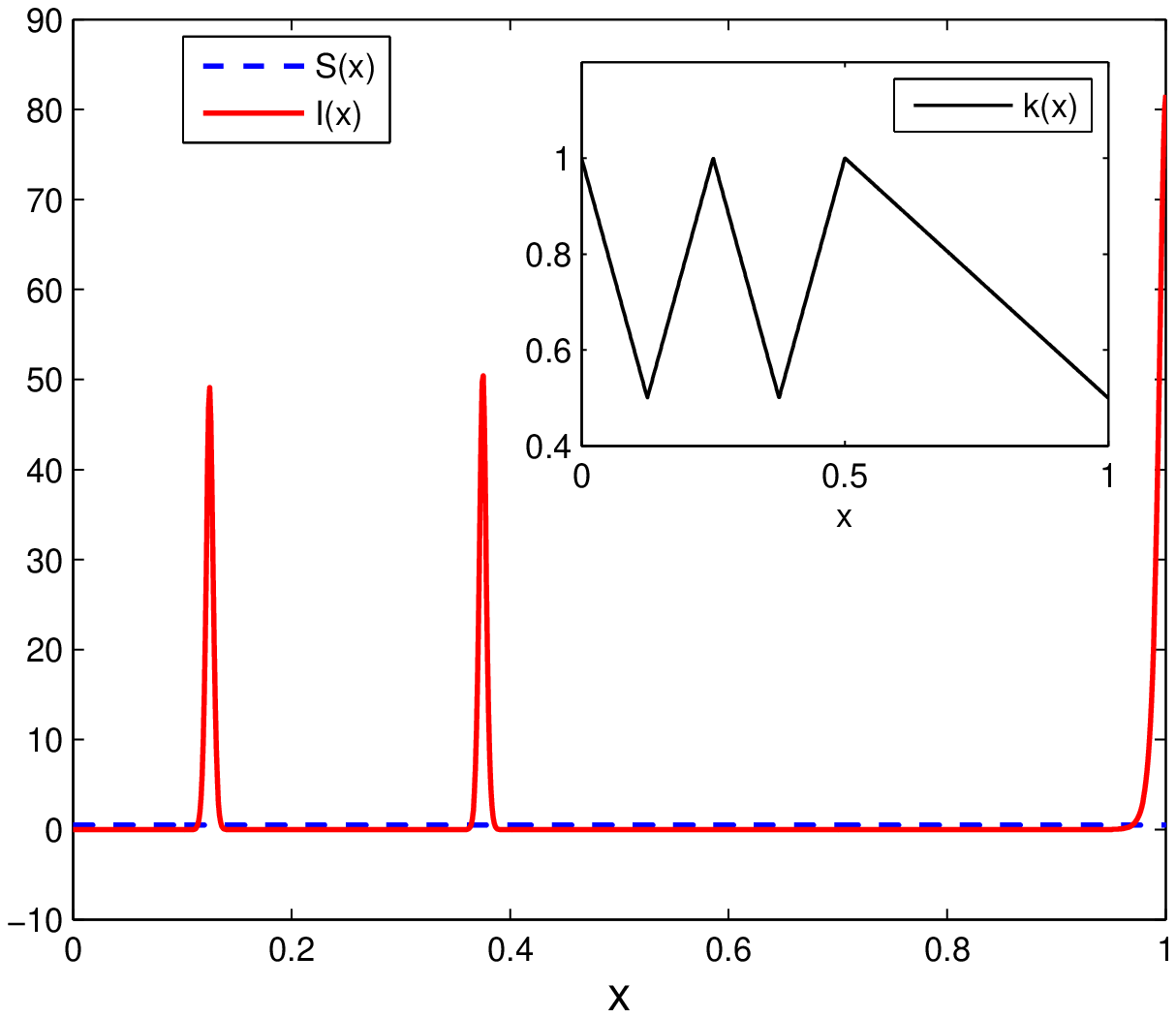}

\small{\hspace{0.6cm} (a)\ $\Theta_k=\{\frac{1}{2}\}$ \hspace{3cm} (b)\ $\Theta_k=\{\frac{1}{8},\,\frac{1}{2}\}$ } \hspace{2cm} (c)\ $\Theta_k=\{\frac{1}{8},\, \frac{3}{8}, 1\}$
\caption{Numerical simulations of the solution profile of  model \eqref{SIS}, where $L=1, N=2, d_S=1, d_I=10^{-7}$, $\beta(x)=1+\frac{1}{2}\sin(2\pi x), \gamma(x)=k(x)\beta(x)$, $k_{\min}=\frac{1}{2}$ and $k(x)$ is chosen as follows. In (a), $k(x)=1+\frac{1}{2}\cos(2\pi x)$. In (b), $k(x)=1-4x,\ 0\leq x<\frac{1}{8}; k(x)=4x,\ \frac{1}{8}\leq x< \frac{1}{4}; k(x)= \frac{3}{2}-2x, \ \frac{1}{4}\leq x<\frac{1}{2}; k(x)=x, \ \frac{1}{2}\leq x\leq1$. In (c), $k(x)=1-4x,\ 0\leq x<\frac{1}{8};\  k(x)= 4x,\ \frac{1}{8}\leq x< \frac{1}{4};\ k(x)=2-4x, \ \frac{1}{4}\leq x <\frac{3}{8}; k(x)=4x-1, \ \frac{3}{8}\leq x <\frac{1}{2};\ k(x)=\frac{3}{2}-x, \ \frac{1}{2}\leq x \leq 1$.}
\label{fig1}
\end{figure}

(iii)\ If $\Theta_{k}$ contains at least one proper interval, then no concentration phenomenon occurs for the disease distribution,
and the infected population will aggregate only on such intervals consisting of highest-risk points, regardless of whether there are isolated highest-risk points or not (see Figure \ref{fig2}(a)-(b)). Indeed, our numerical results indicate that the infected population will aggregate on all such intervals consisting of highest-risk points (see Figure \ref{fig2}(c)); however the population number of $I$ at each such interval may be different, depending on the functions $\beta,\, \gamma$.

\begin{figure}[h]
\centering
\includegraphics[width=5.3cm]{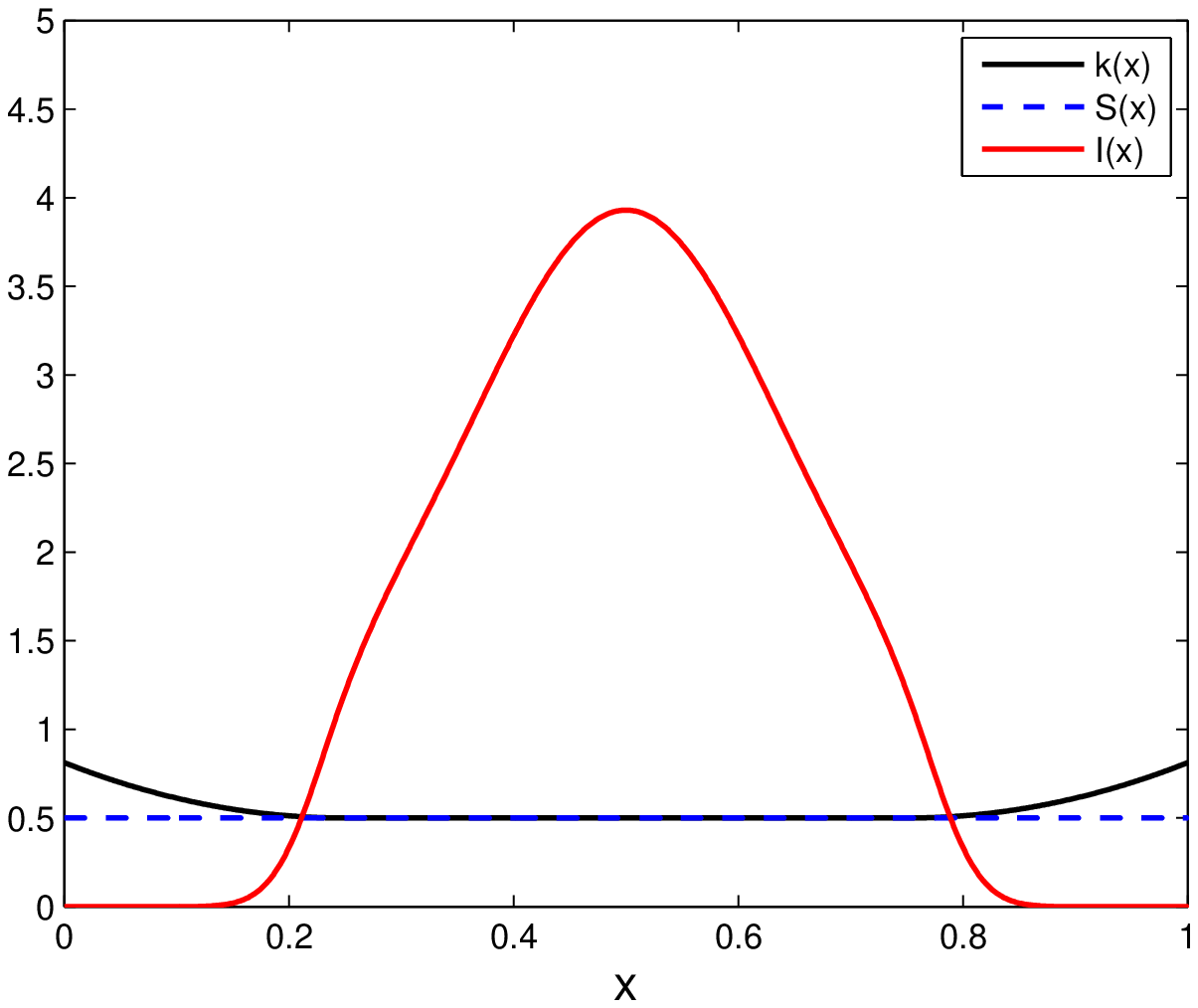}
\includegraphics[width=5.3cm]{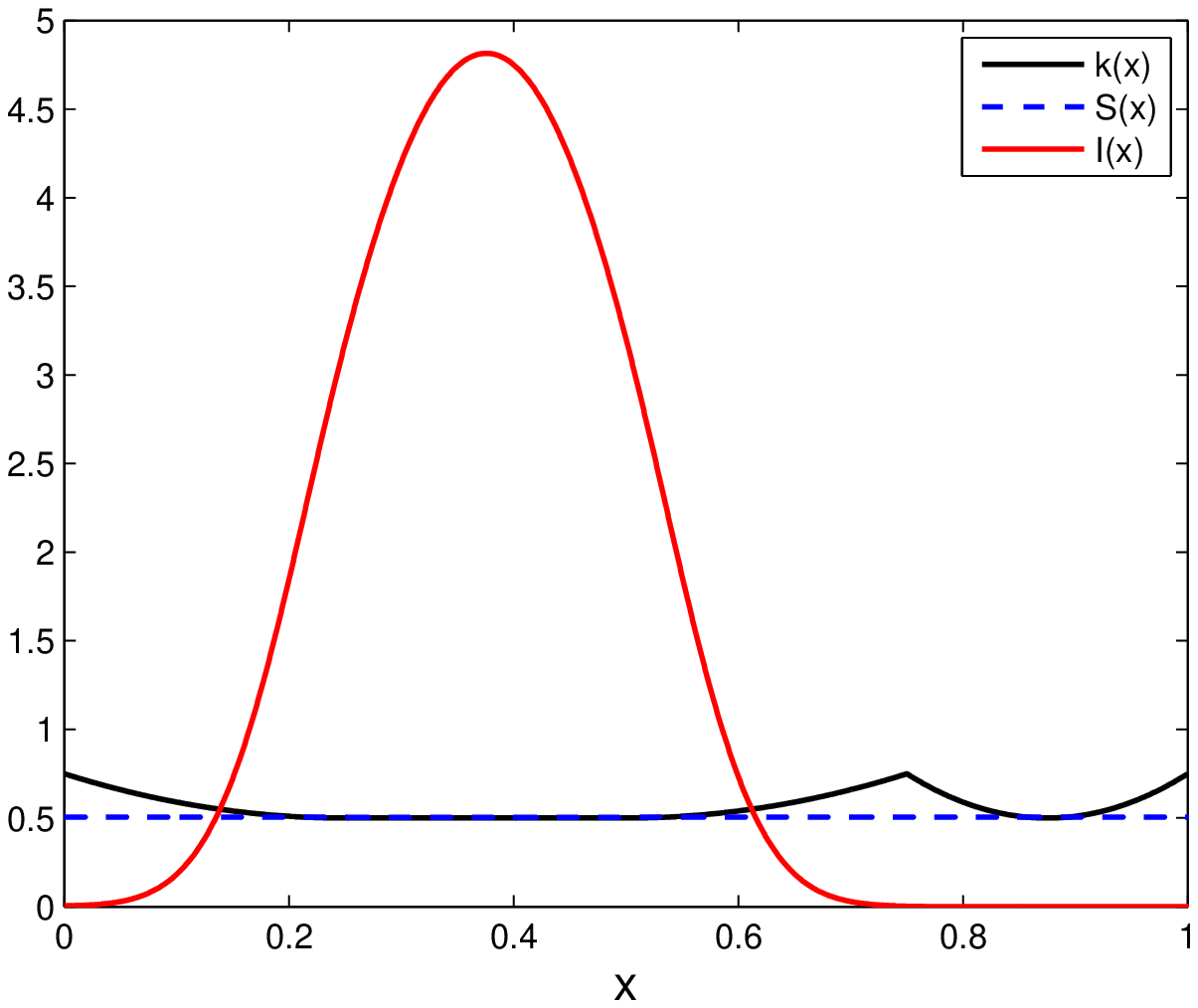}
\includegraphics[width=5.3cm]{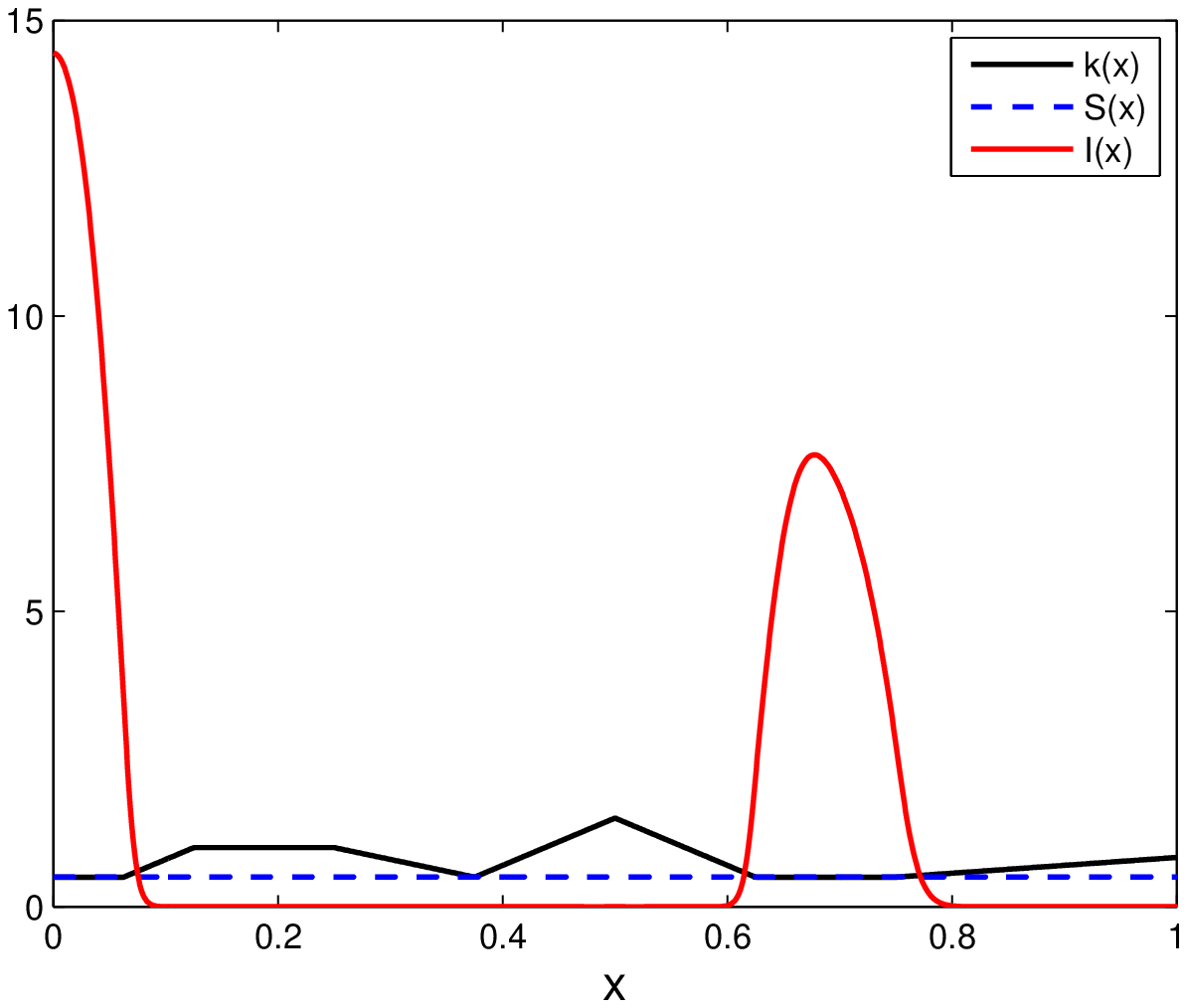}


\scriptsize{\hspace{1cm} (a) $\Theta_k=[\frac{1}{4}, \frac{3}{4}]$ \hspace{2.5cm} (b) $\Theta_k=[\frac{1}{4}, \frac{1}{2}]\cup\{\frac{7}{8}\}$ \hspace{2cm} (c) $\Theta_k=[0, \frac{1}{16}]\cup\{\frac{3}{8}\}\cup[\frac{5}{8}, \frac{3}{4}]$}

\caption{Numerical simulations of the solution profile of  model \eqref{SIS}, where $L=1, N=2, d_S=1, d_I=10^{-5}$, $\beta(x)=1, \gamma(x)=k(x)\beta(x)$, $k_{\min}=\frac{1}{2}$ and $k(x)$ is chosen as follows.  In (a),\ $\Theta_k=[\frac{1}{4}, \frac{3}{4}]$,
 $k(x)=\frac{1}{2}+5(x-\frac{1}{4})^2, 0\leq x<\frac{1}{4};\ k(x)=\frac{1}{2},\ \frac{1}{4}\leq x< \frac{3}{4};\
 k(x)=\frac{1}{2}+5(x-\frac{3}{4})^2, \ \frac{3}{4}\leq x\leq1$.
 \ In (b),\ $\Theta_k=[\frac{1}{4}, \frac{1}{2}]\cup\{\frac{7}{8}\}$, $k(x)=\frac{1}{2}+4(x-\frac{1}{4})^2, \ 0\leq x<\frac{1}{4};\
 k(x)=\frac{1}{2},\ \frac{1}{4}\leq x< \frac{1}{2};\ k(x)=\frac{1}{2}+4(x-\frac{1}{2})^2,\ \frac{1}{2}\leq x <\frac{3}{4};\
 k(x)=\frac{1}{2}+16(x-\frac{7}{8})^2,\ \frac{3}{4}\leq x \leq 1$.
 \ In (c), \ $\Theta_k=[0, \frac{1}{16}]\cup\{\frac{3}{8}\}\cup[\frac{5}{8}, \frac{3}{4}]$, $k(x)=\frac{1}{2},\ 0\leq x<\frac{1}{16};\
 k(x)=8x, \ \frac{1}{16}\leq x< \frac{1}{8};\ k(x)=1, \ \frac{1}{8}\leq x<\frac{1}{4};\ k(x)=2-4x, \ \frac{1}{4}\leq x<\frac{3}{8};\
 k(x)=8x-\frac{5}{2}, \ \frac{3}{8}\leq x< \frac{1}{2};\ k(x)=\frac{11}{2}-8x, \ \frac{1}{2}\leq x<\frac{5}{8};\
 k(x)= \frac{1}{2},\ \frac{5}{8}\leq  x<\frac{3}{4};\ k(x)=\frac{2}{3}x,\ \frac{3}{4}\leq x\leq 1$.}
 \label{fig2}
 \end{figure}

\subsection{Profile of EE of model \eqref{SIS-2} as $d_I\to0$.}
For model \eqref{SIS-2}, for the general H\"{o}lder continuous risk function $h$, under the condition \eqref{cond}, as $d_I\to0$, we know from Theorem \ref{th2.2} that the susceptible population $S$ converges to a positive function $\hat S$, which is non-constant unless $h$ is constant. The infected population $I$ converges to a positive Radon measure $\mu$, whose support is contained in the region in which $\hat S$ touches $h$; in other words, the disease stays only within the place where the susceptible population distributes along the risk function. If the risk function $h$ is of $C^2$, we see from Lemma \ref{l2.3} and Lemma \ref{l2.3a} that the infected population aggregates at least in a neighborhood of the highest-risk locations.
\begin{figure}[h]
\centering
\includegraphics[width=5.3cm]{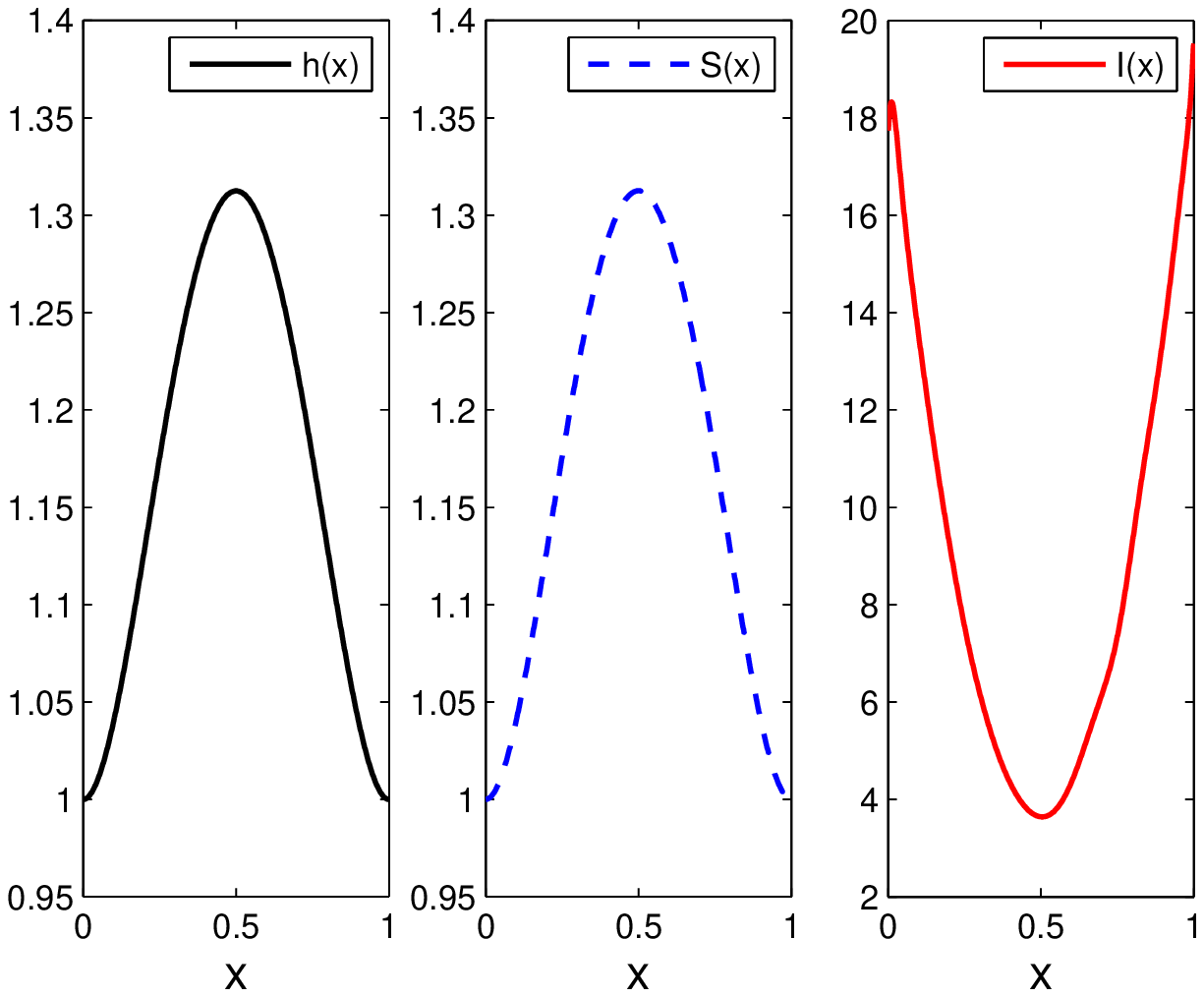}
\includegraphics[width=5.3cm]{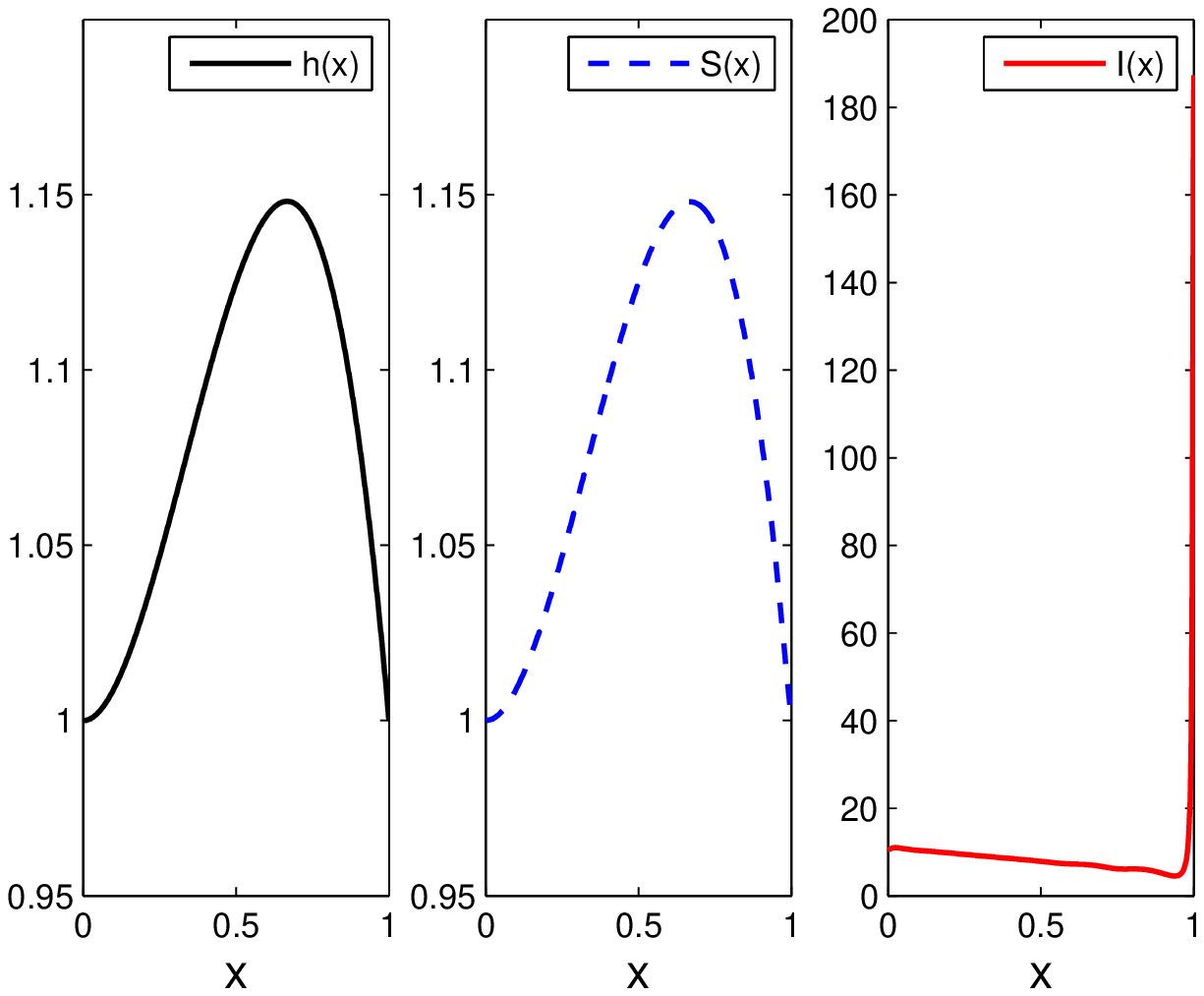}
\includegraphics[width=5.3cm]{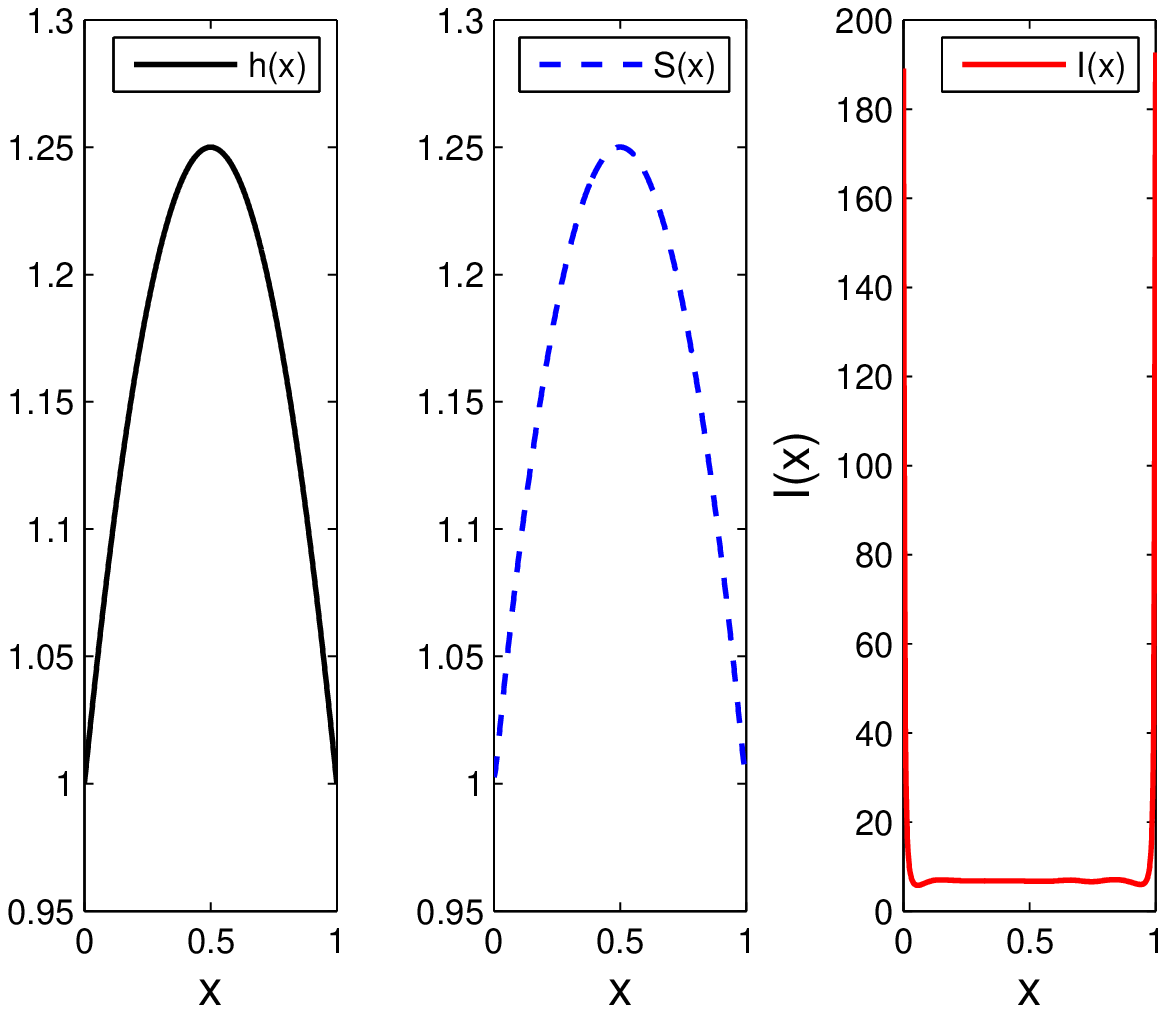}

\small{\ (a) $h(x)=1+5x^2(1-x)^2$ \hspace{1.0cm} (b) $h(x)=1+x^2(1-x)$ \hspace{1.2cm} (c)\ $h(x)=1+x(1-x)$}
\caption{Numerical simulations of the solution profile of  model \eqref{SIS-2}, where $\beta(x)=1+\frac{1}{2}\sin(2\pi x), \eta(x)=1, \gamma(x)=h(x)\beta(x)-\eta(x)$, $d_S=1, d_I=10^{-8},\Lambda=10$. In (a), $h(x)=1+5x^2(1-x)^2$, in (b), $h(x)=1+x^2(1-x)$, and in (c), $h(x)=1+x(1-x)$. }
 \label{fig3}
\end{figure}

Furthermore, when $h\in C^2([0,L])$, in light of Theorem \ref{th2.3}, one can draw the following conclusions concerning the asymptotic profile of $I$.

\vskip6pt
(i)\ For any risk function $h$ satisfying $-d_Sh_{xx}\leq\Lambda-h$ in $(0,L)$, $h_x(0)\geq0,\,h_x(L)\leq0$, and condition \eqref{cond} (for instance, $h<\Lambda$ is a positive constant), the infected population must occupy the entire habitat, and it also forms the concentration phenomenon at the boundary point $x=0$ (or $x=1$) if $h_x(0)>0$ (or $h_x(1)<0$), which is also the highest-risk location; see Theorem \ref{th2.3}(i) and the numerical illustrations in Figure \ref{fig3}(a)-(b)-(c).

\vskip6pt
(ii)\ For any convex risk function $h$ (i.e., $h_{xx}\geq,\not\equiv 0$ on $[0,L]$) fulfilling \eqref{cond}, the infected population usually stays only in part of the habitat. In particular, by  Theorem \ref{th2.3}(ii)(iii),  we can observe the following behaviors.

\vskip6pt
(ii-a)\ If the highest-risk set $\Theta_h$ contains only one point, denoted by $\tau_0$,  then the distribution behavior of the infected population is affected by whether $\tau_0$ is a boundary point or an interior point. More precisely, when $\tau_0$ is an interior point, then the infected population resides in a certain left neighborhood of $\tau_0$, staying away from the boundary points $x=0$ and $x=1$. In fact, such a neighborhood can be calculated through the formula \eqref{th2.3-g}. One may further refer to Figure \ref{fig4}(a).

However, if $\tau_0$ is a boundary point, say $\tau_0=L$, then the infected population stays in a certain neighborhood of $L$ provided $\frac{e^{2Ld_S^{-1/2}}-1}{e^{2Ld_S^{-1/2}}+1}>-\frac{d_S^{1/2}h_x(L)}{\Lambda-h(L)}$, while the infected population concentrates only at $L$ provided $\frac{e^{2Ld_S^{-1/2}}-1}{e^{2Ld_S^{-1/2}}+1}\leq-\frac{d_S^{1/2}h_x(L)}{\Lambda-h(L)}$. Since $h_x(L)\leq0$ in this situation, the infected population stays in a certain neighborhood of $L$ provided for all $d_S>0$ if $h_x(L)=0$. If $h_x(L)<0$, it should be noted that
the function $q(d_S)=d_S^{-1/2}\frac{e^{2Ld_S^{-1/2}}-1}{e^{2Ld_S^{-1/2}}+1}+\frac{h_x(L)}{\Lambda-h(L)}$ deceases in $d_S\in(0,\infty)$, $\lim_{d_S\to0}q(d_S)=\infty$ and $\lim_{d_S\to\infty}q(d_S)=\frac{h_x(L)}{\Lambda-h(L)}<0$. As a result, there is a unique $d_S^*>0$ such that $q(d_S^*)=0$, and in turn the infected population stays in a left neighborhood of $L$ for $0<d_S<d_S^*$ , and the infected population concentrates only at $L$ for all $d_S\geq d_S^*$.

\vskip6pt
(ii-b)\ If the highest-risk set $\Theta_h$ contains only an interval, then the infected population resides in a certain neighborhood of such an interval. Again, such a neighborhood can be calculated through the formula \eqref{th2.3-g}. See the numerical simulation in Figure \ref{fig4}(b).

\vskip6pt
(ii-c) For a general H\"{o}lder continuous risk function $h$, we can conclude that the disease must exist in all isolated highest-risk point(s) and a neighborhood of each highest-risk interval if exists; nevertheless, it is challenging to give a precise characterization for the distribution behavior of the susceptible and infected populations, due to the mathematical difficulties on the analysis of the free boundary problem \eqref{th2.2-1}. We have performed the numerical simulations in Figure \ref{fig5}(a)-(b) as an illustration.

\begin{figure}[h]
\centering
\includegraphics[width=7cm,height=5cm]{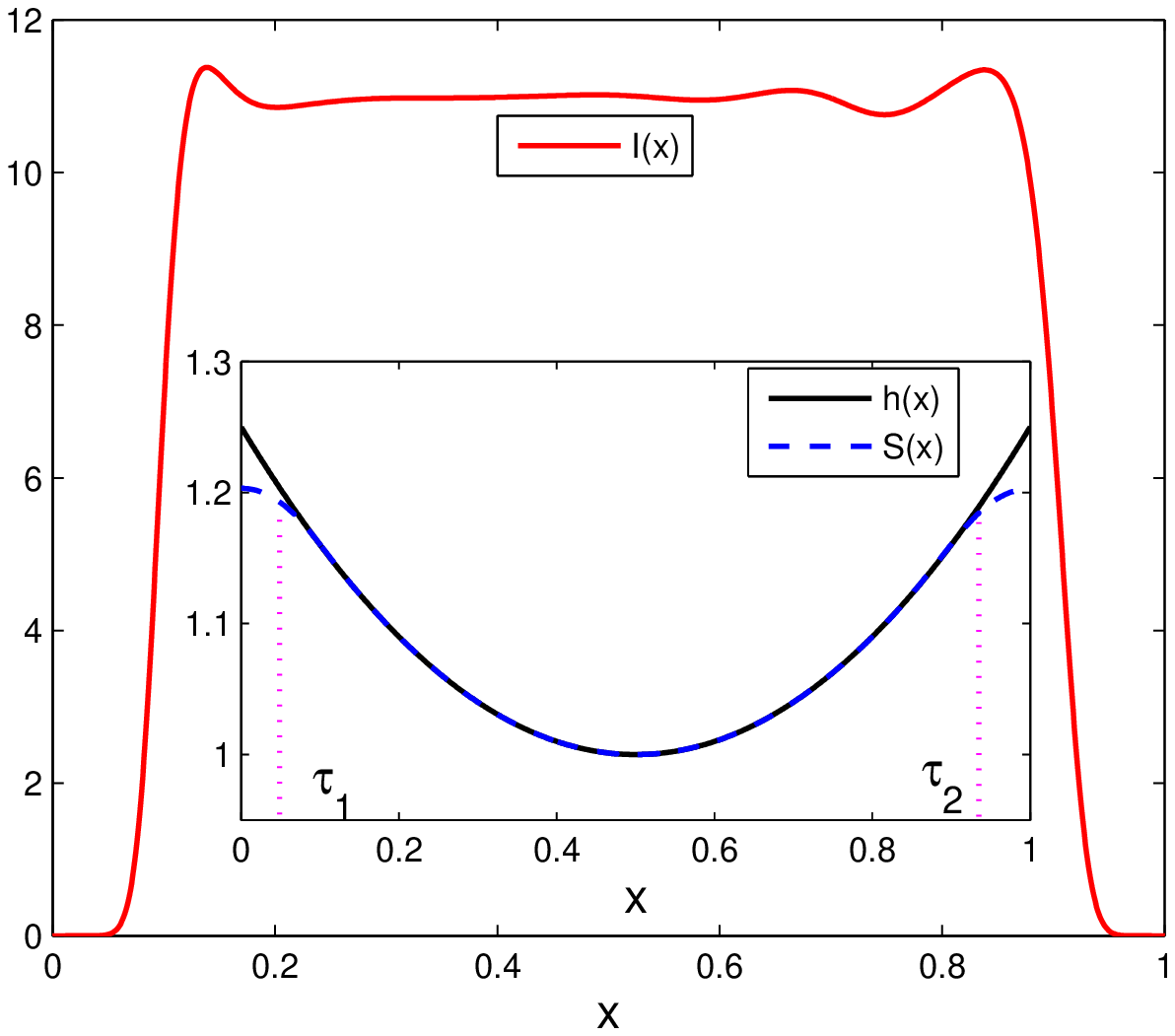}
\includegraphics[width=7cm,height=5cm]{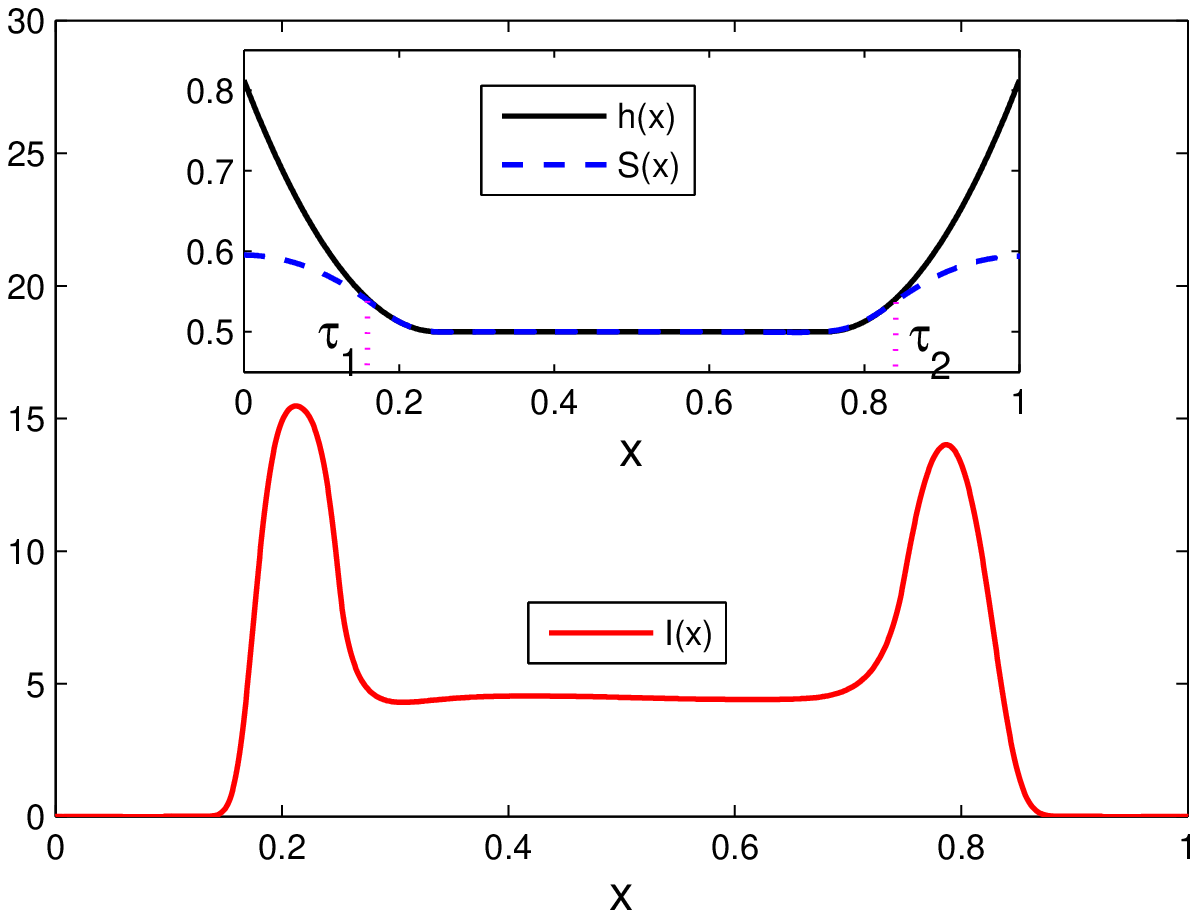}

\small{(a)\ \ $\Theta_h=\{\frac{1}{2}\}$ \hspace{5cm} (b)\ \ $\Theta_h=[\frac{1}{4},\,\frac{3}{4}]$}
\caption{Numerical simulations of the solution profile of  model \eqref{SIS-2}, where $\beta(x)=1+\frac{1}{2}\sin(2\pi x), \eta(x)=1, \gamma(x)=h(x)\beta(x)-\eta(x)$, $d_S=1, d_I=10^{-10},\Lambda=10$, and $h(x)=1+(x-\frac{1}{2})^2$ in (a), while in (b), $h(x)=\frac{1}{2}+5(x-\frac{1}{4})^2,\ 0\leq x<\frac{1}{4};\ h(x)=\frac{1}{2}, \ \frac{1}{4}\leq x< \frac{3}{4};\ h(x)=\frac{1}{2}+5(x-\frac{3}{4})^2, \ \frac{3}{4}\leq x<1$. }
\label{fig4}
\end{figure}

\begin{figure}[h]
\centering
\includegraphics[width=7cm,height=5cm]{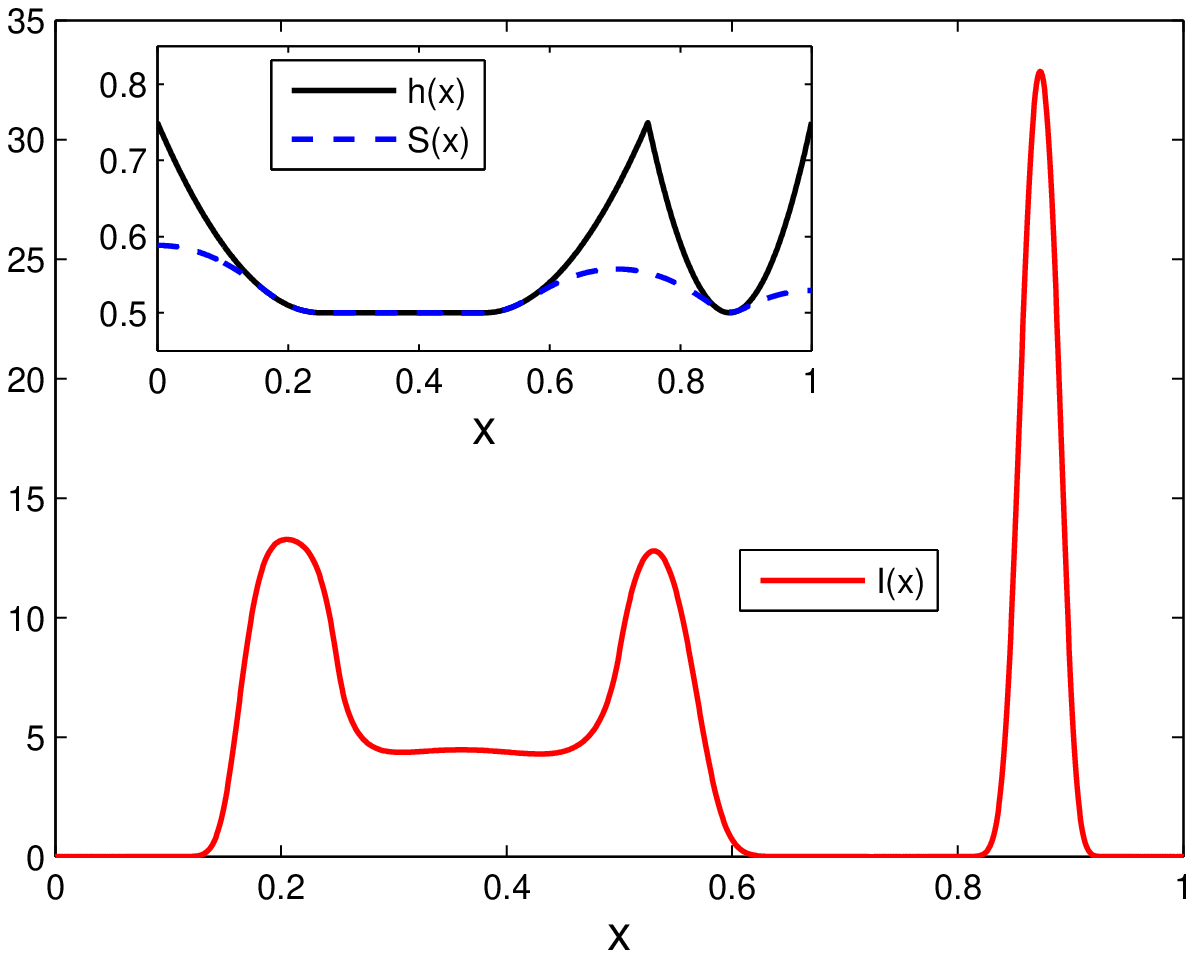}
\includegraphics[width=7cm,height=5cm]{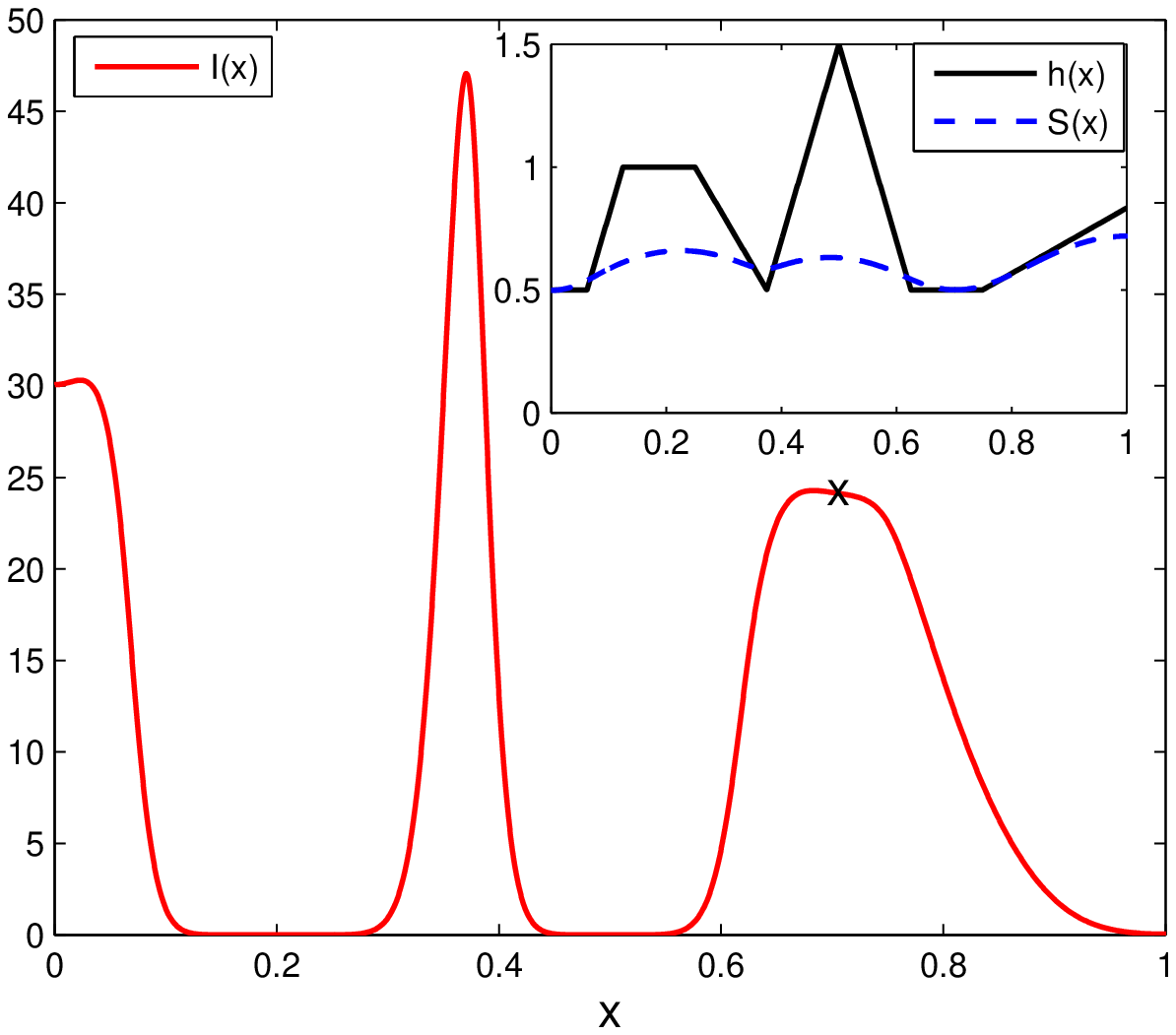}

\small{\hspace{1cm} (a)\ \ $\Theta_h=[\frac{1}{4}, \frac{1}{2}]\cup\{\frac{7}{8}\}$ \hspace{3cm} (b)\ \ $\Theta_h=[0, \frac{1}{16}]\cup\{\frac{3}{8}\}\cup[\frac{5}{8}, \frac{3}{4}]$}
\caption{Numerical simulations of the solution profile of  model \eqref{SIS-2}, where $d_S=1, d_I=10^{-5}, \beta(x)=1+\frac{1}{2}\sin(2\pi x), \eta(x)=1, \gamma(x)=h(x)\beta(x)-\eta(x)$, $\Lambda=10$.  In (a) and (b), $h(x)$ is chosen to be the same as $k(x)$ in Figure \ref{fig2}(b) and Figure \ref{fig2}(c), respectively.}
 \label{fig5}
\end{figure}

\vskip8pt
In what follows, we would like to make some more discussions on (ii-a) above in the case that $\tau_0$ is a boundary point. For example, we take $\tau_0=L$, and also assume that $h_x(L)<0$. On the one hand, by fixing $h_x(L)$, we have known from (ii-b) that large diffusion rate $d_S$ can result in the disease concentration only at the location $L$ and small diffusion rate $d_S$ will cause the disease to distribute in a left neighborhood of $L$. On the other hand, once $d_S$ is fixed,
the concentration phenomenon happens only if $-h_x(L)$ is properly large. This motivates us to see whether a similar concentration phenomenon could occur at an interior isolated highest-risk point if the risk function $h$ is merely H\"{o}lder continuous. To illustrate this phenomenon, let us consider the following risk function whose curve is the connection of two segments:
\begin{equation}\label{f6}
h(x)=\left\{\begin{array}{ll}
a_{1}\left(x-\frac{L}{2}\right)+\frac{\Lambda}{4}, & x \in\left[0, \frac{L}{2}\right], \\[2mm]
a_{2}\left(x-\frac{L}{2}\right)+\frac{\Lambda}{4}, & x \in\left(\frac{L}{2}, L\right],
\end{array}\right.
\end{equation}
with $a_1<0, a_2>0$. Obviously, $h$ is merely Lipschitz continuous at $x=\frac{L}{2}$.
Our numerical simulation results demonstrate that if the slopes $|a_1|,\,a_2$ are properly large, then the infected population will concentrate at $x=\frac{L}{2}$ (Figure \ref{fig6}(a)); if $|a_1|,\,a_2$ are small, then the infected population will aggregate in a neighborhood of $x=\frac{L}{2}$ (Figure \ref{fig6}(b)); and if
$|a_1|$ is small while $a_2$ is large, then the infected population will aggregate in a left-neighborhood of $x=\frac{L}{2}$ (Figure \ref{fig6}(b)). These profiles behave rather differently from that in Theorem \ref{th2.3}(ii) for $h\in C^2([0,L])$, as shown by Figure \ref{fig4}(a). Therefore, the numerical results reveal that the smoothness of $h$ may have a substantial effect on the spatial distribution of the disease.
\begin{figure}[h]
\centering
\includegraphics[width=5.3cm]{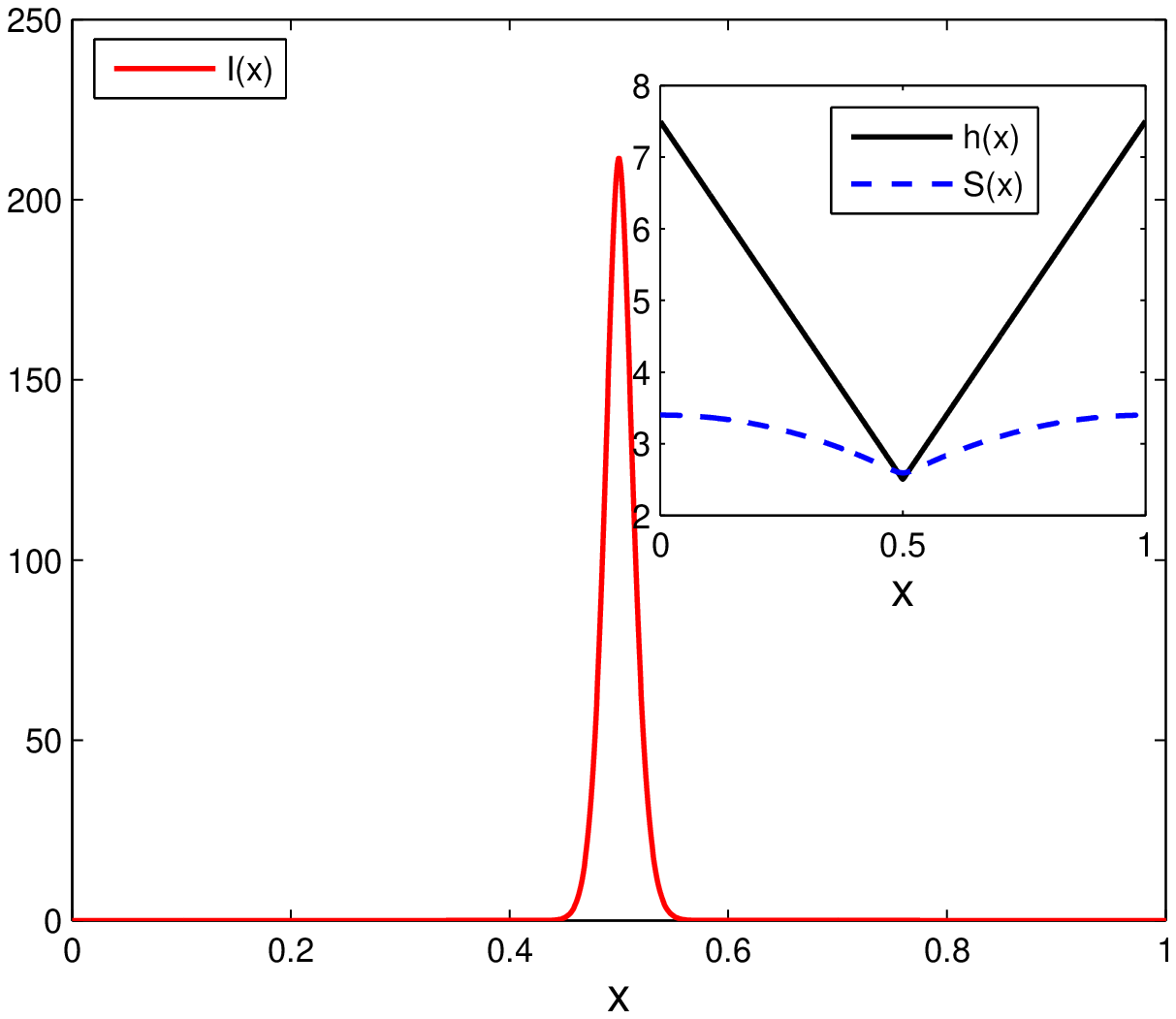}
\includegraphics[width=5.3cm]{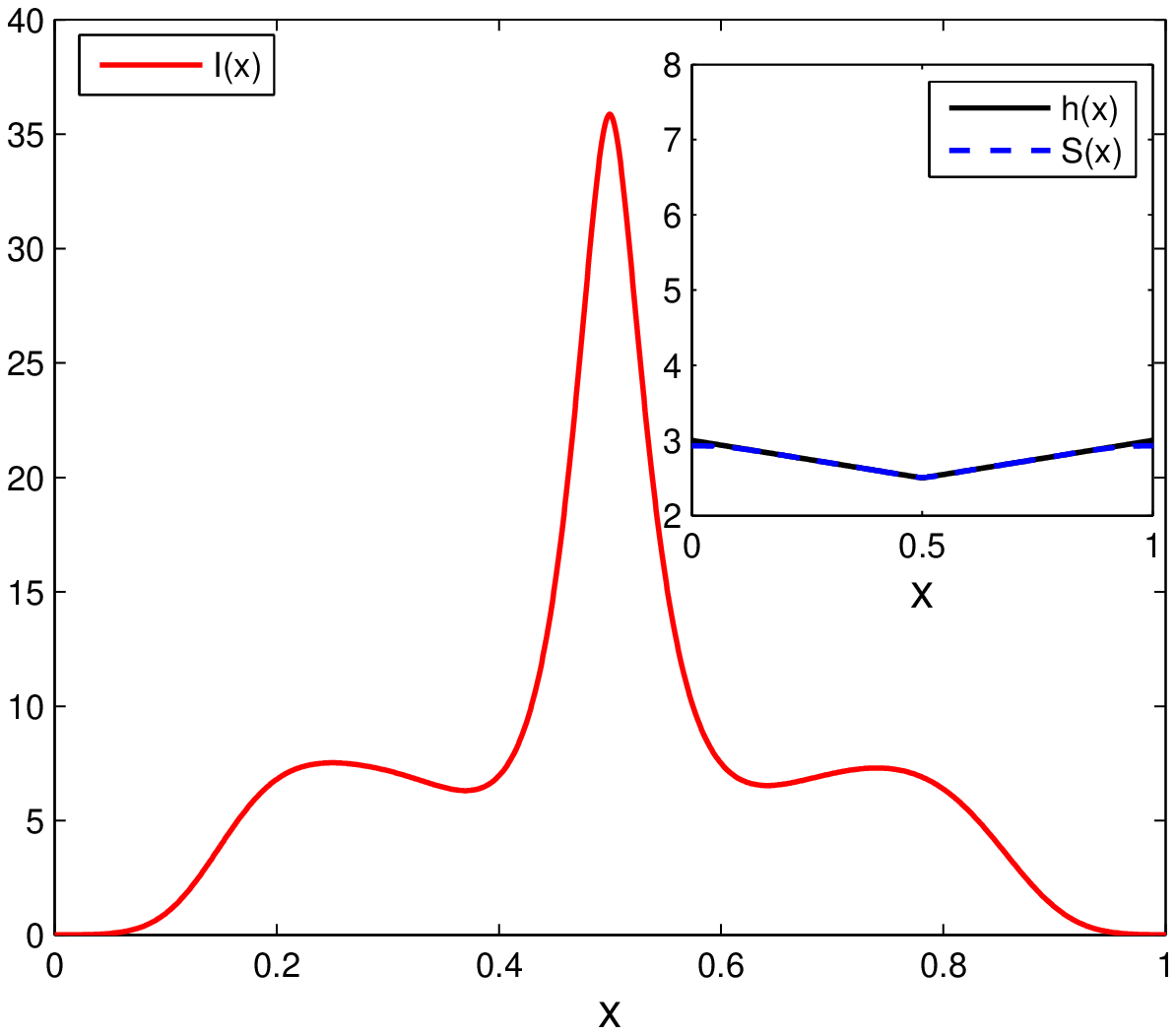}
\includegraphics[width=5.3cm]{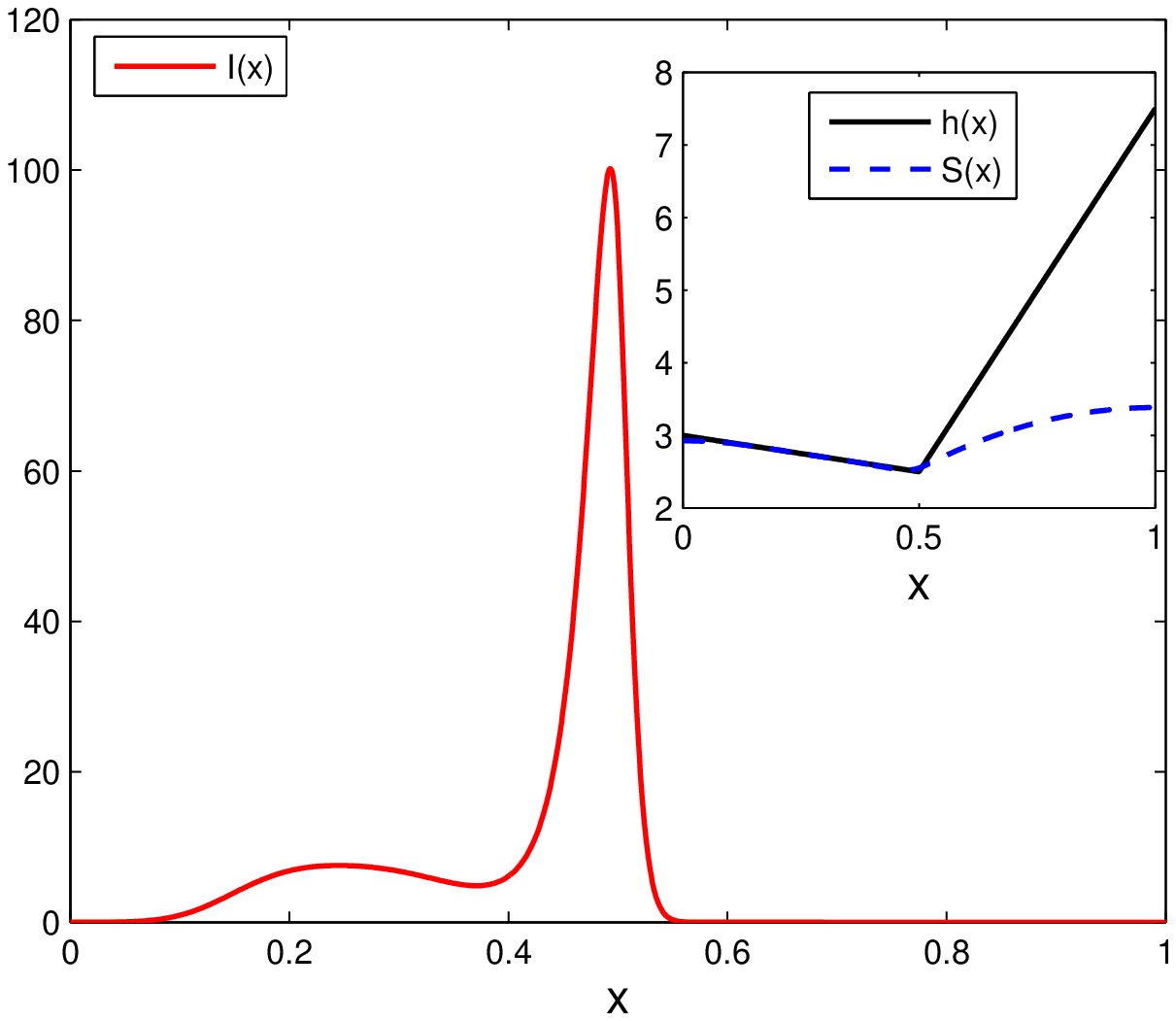}

\small{(a) $a_1=-10,a_2=10$  \hspace{2cm} (b) $a_1=-1, a_2=1$ \hspace{2cm} (c)\ $a_1=-1, a_2=10$}
%
%
\caption{Numerical simulations of the solution profile of  model \eqref{SIS-2}, where $\beta(x)=1+\frac{1}{2}\sin(2\pi x), \eta(x)=1, \gamma(x)=h(x)\beta(x)-\eta(x)$, $L=1, d_S=1, d_I=10^{-5}, \Lambda=10$ and $h(x)$ is given by \eqref{f6}.}
 \label{fig6}
\end{figure}

\subsection{Conclusion.} The discussions in the above two subsections, together with the numerical simulations, show that the spatial profile of the susceptible and infected populations of \eqref{SIS} and \eqref{SIS-2} with respect to small movement rate of the infected individuals are rather different. This is caused by the presence of the recruitment term for the susceptible population and the death rate for the infected population. On the other hand, we would like to mention that the recent works \cite{Cui,CLL,CLPZ,CL,KMP,LZ,ZC} studied various kinds of reaction-diffusion-advection SIS epidemic models, in which the advection term represents some passive movement in a certain direction, e.g., due to external environmental forces such as water flow \cite{Lutscher1, Lutscher2, Lutscher3,PZ2}, wind \cite{Dahmen} and so on. In particular, if an advection is present in \eqref{SIS} and stands for, for instance, the water flow, it was proved in \cite[Theorem 1.4]{CLPZ} that, as $d_I\to0$, the susceptible population converges to a positive function while the infected population concentrates only at the downstream of the water flow; a similar result can be shown to hold for the corresponding system \eqref{SIS-2}. Such a distribution behavior is essentially different from that of \eqref{SIS} and \eqref{SIS-2} with  small $d_I$.

In summary, our results here, combined with those of \cite{CLPZ,KMP,LPW1}, suggest that the recruitment term for the susceptible population, the death rate for the infected population (even the smoothness of the associated risk function) as well as the advection can lead to significant impacts on the disease transmission and thus decision-makers should attach great importance to these factors when taking measures such as the lockdown and quarantine to control the movement or immigration of the infected individuals so as to eliminate the disease infection.

\section{Appendix}

In this appendix, we always let $\Omega$ be a smooth and bounded domain in $\mathbb R^n\,(n\geq1)$.
Given $f\in C(\overline\Omega)$, consider the following eigenvalue problem with Neumann boundary condition:
 \begin{equation}
  \label{eq:1}
\begin{cases}
  -D\Delta \phi+f(x)\phi=\lambda\phi &\hbox{ in }\Omega,\\
  \frac{\partial\phi}{\partial \nu}=0 &\hbox{ on }\partial\Omega,
\end{cases}
\end{equation}
where $\nu(x)$ is the unit exterior normal vector of $\partial\Omega$ at $x$, and the coefficient $D$ is a positive constant.

We start with a well-known fact concerning the asymptotic behavior of the principal eigenvalue of \eqref{eq:1}
with respect to small diffusion; one may refer to, for example, \cite[Lemma 3.1]{LN2006}.

\begin{lem}\lbl{a2.1} Let $\lambda_1(D,f)$ be the principal eigenvalue of \eqref{eq:1}. Then it holds that
 $$
 \lim_{D\to0}\lambda_1(D,f)=\min_{x\in\overline\Omega}f(x).
 $$
\end{lem}

We next recall the $L^1$-estimate for the weak solution (due to \cite{BS}) of the following linear elliptic problem:
 \bes
 -\Delta w+c(x)w=g \ \ \ \mbox{in}\ \Omega,\ \ \ \
 {{\partial w}\over{\partial\nu}}=0\ \ \mbox{on}\ \partial\Omega.
 \lbl{aa2.1}
 \ees

\begin{lem}\lbl{a2.2} {\rm (a)}\ {\rm (Global estimates)} \ Assume that $c\in L^\infty(\Omega),\ g\in L^1(\Omega)$ and let
$w\in W^{1,1}(\Omega)$ be a weak solution of {\rm(\ref{aa2.1})}. Then, for any $r\in[1, n/{(n-1)})$, we have
$w\in W^{1,r}(\Omega)$  and the following estimate
$$
\|w\|_{W^{1,r}(\Omega)}\le C\|g\|_{L^1(\Omega)},
$$
where the positive constant $C$ is independent of $w$.

{\rm (b)\ (Interior estimates)}\ Assume that $\Omega'\subset\subset\Omega$ is a smooth domain, $c\in L^\infty(\Omega),\ g\in L^1(\Omega)$, and let $w\in W^{1,1}(\Omega)$ be a weak solution to the equation $-\Delta w+c(x)w=g$. Then, for any $r\in[1, n/{(n-1)})$, we have $w\in W^{1,r}(\Omega')$   and the following estimate
 $$\|w\|_{W^{1,r}(\Omega')}\le C\|g\|_{L^1(\Omega)},$$ where the positive constant $C$ is independent of $w$.

 \end{lem}

At last, we state a Harnack-type inequality for weak solutions (see, e.g.,
\cite{Lie} or \cite{PSW}), whose strong form was obtained in \cite{LNT}.

\begin{lem}\lbl{a2.3} {\rm (a)}\ {\rm (Global Harnack inequality)} \ Let $c\in
L^r(\Omega)$ for some $r>n/2$. If $w\in W^{1,2}(\Omega)$ is a
non-negative weak solution of the boundary value problem
 \bes
 -\Delta w+c(x)w=0\ \ \mbox{in}\ \Omega,\ \ \ \ \
 {{\partial w}\over{\partial\nu}}=0\ \ \mbox{on}\ \partial\Omega,
 \nonumber
 \ees
then there is a constant $C$, determined only by $\|c\|_r,\,r$ and
$\Omega$ such that $$\sup_\Omega\,w\leq C\inf_\Omega\,w.$$

{\rm (b)}\ {\rm (Local Harnack inequality)} \ Let $\Omega'\subset\subset\Omega$ be a smooth domain and $c\in
L^r(\Omega)$ for some $r>n/2$. If $w\in W^{1,2}(\Omega)$ is a
non-negative weak solution of the equation $-\Delta w+c(x)w=0$,
then there is a constant $C$, determined only by $\|c\|_r,\,r,\,\Omega$ and
$\Omega'$, such that $$\sup_{\Omega'}\,w\leq C\inf_{\Omega'}\,w.$$

\end{lem}

\end {document}